\documentclass{article}

\usepackage{geometry}
\usepackage{cancel}
\usepackage{graphicx}
\usepackage{subcaption}
\usepackage{adjustbox}

\usepackage{algorithmicx}
\usepackage[ruled,vlined]{algorithm2e}
\usepackage{algpseudocode,float}
\algblock{ParFor}{EndParFor}

\usepackage{mathrsfs}  
\usepackage{sidecap}
\usepackage{amsfonts}
\usepackage{amssymb}
\usepackage{amsmath}

\usepackage{bm}

\usepackage{amsthm}
\usepackage{comment}
\usepackage{xcolor}
\usepackage{relsize}
\usepackage{multirow}
\usepackage{multicol}
\usepackage{diagbox}
\usepackage{overpic}
\usepackage{authblk}

\newcommand{\R}{{\mathbb R}}

\usepackage{amssymb}

\usepackage{cleveref}
\numberwithin{equation}{section}

\DeclareMathOperator*{\argmin}{arg\,min}

\newtheorem{proposition}{Proposition}
\newtheorem{definition}{Definition}

\date{}

\title{Whiteness-based parameter selection for Poisson data in variational image processing}

\newcommand{\D}{\mathrm{D}}

\newcommand{\I}{\mathrm{I}}

\usepackage{subcaption}
  \definecolor{cosmiclatte}{rgb}{1.0, 0.97, 0.91}
\definecolor{azure}{rgb}{0.94, 1.0, 1.0}
\definecolor{beaublue}{rgb}{0.74, 0.83, 0.9}
\usepackage{color, colortbl}
\definecolor{Gray}{rgb}{0.75,0.75,0.75}
\definecolor{apricot}{rgb}{0.98, 0.81, 0.69}
\definecolor{almond}{rgb}{0.94, 0.87, 0.8}
\definecolor{LightCyan}{rgb}{0.88,1,1}

\usepackage{autonum}

\author[1]{Francesca Bevilacqua \thanks{francesca.bevilacqu8@unibo.it}}
\author[1]{Alessandro Lanza \thanks{alessandro.lanza2@unibo.it}}
\author[2]{Monica Pragliola\thanks{monica.pragliola@unina.it}}
\author[1]{Fiorella Sgallari \thanks{fiorella.sgallari@unibo.it}}
\affil[1]{Department of Mathematics, University of Bologna, Bologna, Italy}
\affil[2]{Department of Mathematics and Applications, University of Naples Federico II, Naples, Italy}

\begin{document}
	
\maketitle

\noindent \textbf{Abstract.} We propose a novel automatic parameter selection strategy for variational imaging problems under Poisson noise corruption. 
The selection of a suitable regularization parameter, whose value is crucial in order to achieve high quality reconstructions, is known to be a particularly hard task in low photon-count regimes. 
In this work, we extend the so-called residual whiteness principle originally designed for additive white noise to Poisson data. The proposed strategy relies on the study of the whiteness property of a standardized Poisson noise process. 
After deriving the theoretical properties that motivate our proposal, we solve the target minimization problem with a linearized version of the Alternating Direction Method of Multipliers (ADMM), which is particularly suitable in presence of a general linear forward operator. Our strategy is extensively tested on image restoration and Computed Tomography (CT) reconstruction problems, and compared to the well-known discrepancy principle for Poisson noise proposed by Zanella at al. and with a nearly exact version of it previously proposed by the authors.

\section{Introduction}
\label{sec:intro}

In many research areas related to imaging applications, such as astronomy, microscopy and computed tomography, the acquired images are formed by counting the number of photons irradiated by a source and hitting the image domain. The number of photons measured by the sensor can differ from the expected one due to fluctuations that are modelled by a Poisson noise \cite{poissbook}.

The general image formation (or degradation) model under Poisson noise corruption in vectorized form reads
\begin{equation}
\bm{y} \;\,{=}\,\; \bm{\mathrm{poiss}}
\left(\,\overline{\bm{\lambda}}\,\right)\,, \quad\;
\overline{\bm{\lambda}} = \bm{g}\left(\bm{\mathrm{H} }\bar{\bm{x}}\right) + \bm{b}\,,
\label{eq:formod}
\end{equation}
%
where $\,\bm{y} \in \mathbb{N}^m$,  $\overline{\bm{x}} \in \R_+^n$ and $\bm{b}\in\R_+^{m}$ - with $\mathbb{N}$ and $\R_+$ denoting the sets of natural numbers including zero and of non-negative real numbers, respectively - are vectorized forms of the observed degraded $m_1 \times m_2$ image, the unknown uncorrupted $n_1 \times n_2$ image and the so-called (usually known) background emission $m_1 \times m_2$ image, respectively, with $m = m_1m_2$, $n = n_1 n_2$.
Matrix $\bm{\mathrm{H}} \in \R^{m \times n}$ contains the coefficients of a linear degradation operator, whereas the vectorial function $\bm{g}:\R^m\to \R^m$ is the identity function or a nonlinear function modelling the eventual presence of (deterministic) nonlinearities in the degradation process. 
$\bm{\mathrm{H}}$ and $\bm{g}$ are determined by the specific application at hand and here are assumed to be known. In particular, for the applications of interest in this paper, the function $\bm{g}$ can be restricted to the simplified form $\,\bm{g}(\bm{h}) = \left(g(h_1),g(h_2),\ldots,g(h_m)\right)^T$, with $g: \R_+ \to \R_+$.
Finally, $\bm{\mathrm{poiss}}
\left(\,\overline{\bm{\lambda}}\,\right) := \left(\mathrm{poiss}\left(\overline{\lambda}_1\right),\mathrm{poiss}\left(\overline{\lambda}_2\right),\ldots,\mathrm{poiss}\left(\overline{\lambda}_m\right)\right)$, with  $\mathrm{poiss}(\,\overline{\lambda_i}\,)$ indicating the realization of a Poisson-distributed random variable with parameter (mean) $\overline{\lambda}_i$, hence $\overline{\bm{\lambda}} \in \R_+^m$ is the vectorized form of the $m_1 \times m_2$ noise-free degraded image.

The inverse problem of determining a good estimate $\bm{x}^*$ of the uncorrupted image $\overline{\bm{x}}$ given the degraded observation $\bm{y}$ is a hard task even when $\bm{\mathrm{H}}$, $\bm{g}$ and $\bm{b}$ are known. In fact, such an inverse problem is typically ill-posed and, hence, some a priori information or belief on the target image $\overline{\bm{x}}$ must necessarily be encoded, e.g. in the form of regularization, in order to obtain an acceptable estimate $\bm{x}^*$. A popular and effective approach allowing to explicitly include regularization is the so-called variational approach, according to which an estimate $\bm{x}^*$ of the original image $\bar{\bm{x}}$ is sought as the global minimizer of a given cost (or energy) function $\mathcal{J}:\R^n\to\R$; in formula
\begin{equation}
\label{eq:varmod}
\bm{x}^*(\mu)\in\argmin_{\bm{x}\,\in\,\R_+^n}\{\,  \mathcal{J}(\bm{x};\mu)\;{:=}\; \mathcal{R}(\bm{x}) \;{+}\; \mu \, \mathcal{F}(\bm{\lambda};\bm{y})\,\}\,,\quad \bm{\lambda} \;{:=}\,\bm{g}(\bm{\mathrm{H}x})+\bm{b}\,,
\end{equation}
where $\mathcal{R}$ and $\mathcal{F}$ are referred to as the \emph{regularization term} and the \emph{data fidelity term}, respectively, and where the so-called regularization parameter $\mu \in \R_{++}$ allows to balance the contribution of the two terms in the overall cost function. 

The data fidelity term $\mathcal{F}$ measures the
discrepancy between the noise-free degraded image $\bm{\lambda}$ and the noisy observation $\bm{y}$ in a way that  accounts for the noise statistics. In presence of Poisson noise, according to the Maximum Likelihood (ML) estimation approach, the fidelity term is typically set as the (generalized) Kullback-Leibler (KL) divergence between $\bm{\lambda}$ and $\bm{y}$ (see, e.g., \cite{nedp}), that is
\begin{equation}
\mathcal{F}(\bm{\lambda};\bm{y})\;{=}\; \mathrm{KL}(\bm{\lambda};\bm{y})\;{:=}\; \sum_{i=1}^m \, \left(\lambda_i-y_i \ln \lambda_i + y_i \ln y_i - y_i\right)\,,
\label{eq:KL}
\end{equation}
where $\,0 \ln 0 = 0\,$ is assumed. We indicate by $\mathcal{R}$-KL the class of variational models defined as in \eqref{eq:varmod} with $\mathcal{F}$ equal to the KL divergence term in \eqref{eq:KL}.

The regularization term $\mathcal{R}$ in \eqref{eq:varmod} encodes prior information or beliefs on the target uncorrupted image $\overline{\bm{x}}$. One of the most popular and widely adopted regularizers in imaging is the  Total Variation (TV) semi-norm \cite{tv}, which reads
\begin{equation}
\mathcal{R}(\bm{x})\;{=}\;\mathrm{TV}(\bm{x}) \;{:=}\;\sum_{i=1}^n \|(\bm{\nabla x})_i\|_2\,,
\label{eq:TV}
\end{equation}
where $(\bm{\nabla x})_i \in \R^2$ denotes the discrete gradient of image $\bm{x}$ at \mbox{pixel location $i$.} 
The TV term is known to be particularly effective for the regularization of piece-wise constant images as it promotes sparsity of gradient magnitudes. 
We denote by TV-KL the $\mathcal{R}$-KL variational model with $\mathcal{R}$ equal to the TV function in \eqref{eq:TV}.



The regularization parameter $\mu$ in \eqref{eq:varmod} is of crucial importance for getting high quality reconstructions $\bm{x}^*(\mu)$. In fact, it is well established that, even upon the selection of suitable fidelity and regularization terms, an incorrect value of $\mu$ can easily lead to meaningless reconstructions.

For this reason, a lot of research has been devoted to the design of effective strategies for the $\mu$-selection task under Poisson noise corruption. In abstract form, such strategies can be formulated as follows:
\begin{equation}
\text{Select}\;\;\mu=\mu^*\;\;\text{such that}\;\;\mathcal{C}(\bm{x}^*(\mu^*))\;\;\text{is satisfied}\,,
\end{equation}
where $\bm{x}^*(\mu):\R_{++}\to\R^{n}$ is the image reconstruction function introduced in \eqref{eq:varmod} and where $\mathcal{C}(\cdot)$ is some selection criterion or principle.

The selection principles designed so far to deal with Poisson noise have mostly been inspired by the very wide literature related to the parameter selection under additive white Gaussian noise corruption, and they can be thus divided into two main classes according to their original derivation set-up:
\begin{enumerate}
	\item principles derived from imposing the value of some $\mu$-dependent quantity; 
	\item principles derived from optimizing some $\mu$-dependent quantity. 
\end{enumerate}
As typical examples of first class strategies, we mention the \emph{discrepancy principles} (DP) whose general form is given by
\begin{equation}\label{eq:discrep}
\mathcal{C}(\bm{x}^*(\mu^*)):\;\quad \mathcal{D}(\mu;\bm{y}) \;{=}\; \Delta\in\R_{++}\,,
\end{equation}
with $\mathcal{D}(\mu;\bm{y})=\mathrm{KL}(\bm{\lambda}(\mu);\bm{y})$, and where the above equality is referred to as \emph{discrepancy equation} while $\Delta$ is the so-called \emph{discrepancy value} that changes when considering different DP instances. 

The Morozov discrepancy principle, which is widely adopted in presence of Gaussian noise, naturally induces a DP version according to which the scalar value $\Delta$ is replaced by the expected value of the KL fidelity term regarded as a function of the $m$-variate random vector $(Y_1,Y_2,\ldots,Y_m)$, with vector $\bm{\lambda}$ fixed. Unfortunately, recovering an exact closed-form expression for the target expected value has been proven to be theoretically unfeasible, especially in low counting regimes, i.e. when the entries of $\bm{\lambda}$ are small \cite{nedp}. 

A popular alternative to the aforementioned exact but theoretical DP, which has been proposed in \cite{bert1} for the image denoising problem and extended in \cite{bert2} to the restoration task, is based on truncating the Taylor series expansion of the theoretical expected value. For this reason, in \cite{nedp} this version of DP has been referred to as Approximate DP (ADP); in formula, it reads
\begin{equation}\label{eq:ADP}
\tag{ADP}
\mathcal{C}(\bm{x}^*(\mu^*)):\;\quad  \mathcal{D}(\mu^*;\bm{y})\;{=}\;\frac{m}{2}\,,
\end{equation}
with $m$ indicating the total number of pixels in the observed image $\bm{y}$.
The ADP is particularly robust from the theoretical viewpoint as it has been proven that the associated discrepancy equation admits a unique solution. Nonetheless, it has also been observed that, as the number of photons hitting the image domain gets smaller, the approximation considered by ADP becomes particularly rough and the principle returns low-quality reconstruction - see, e.g., \cite{nedp}. 

In order to overcome the limitations presented by the ADP, in \cite{nedp} the authors proposed the Nearly Exact DP (NEDP), which is a novel version of DP based on a more accurate approximation - both in low-counting and mid/high-counting regimes - of the expected value of the KL divergence term. In the NEDP, the discrepancy value $\Delta$ is replaced by a weighted least-square fitting of Montecarlo realizations of the target expected value and it is regarded as a function $f$ of the regularization parameter $\mu$; in formula:
\begin{equation}\label{eq:NEDP}
\tag{NEDP}
\mathcal{C}(\bm{x}^*(\mu^*)):\;\quad  \mathcal{D}(\mu;\bm{y})\;{=}\;f(\mu)\,.
\end{equation}
Despite its very good experimental performances, the NEDP is characterized by theoretical limitations which are mostly related to the lack of guarantees on the uniqueness of the solution for the discrepancy equation; such limitations are also combined with the empirical evidence of multiple solutions in very extreme scenarios where the number of zero-pixels in the acquired data is particularly relevant.

The general formulation of minimization-based principles is
\begin{equation}
\mathcal{C}(\bm{x}^*(\mu^*))\;:\;\quad \mu^*\in\argmin_{\mu\in\R_{++}} \mathcal{V}(\mu)\,,\quad \mathcal{V}:\R_{++}\to\R
\end{equation}
where $\mathcal{V}$ represents some demerit function to be minimized for selecting $\mu$.


For Poisson data, this class of strategies has not been explored as much as the former. A few decades ago, some attempts have been made in order to adapt the popular Generalized Cross Validation (GCV) approach \cite{gcv} to non-Gaussian data \cite{GCVP1,GCVP2}; nonetheless, these strategies, which ultimately rely on a weighted approximation of the KL fidelity term and on a slight reformulation of the classical GCV score, have not been diffusively employed for imaging problems.

Among the parameter selection strategies that have been developed in the context of additive white noise corruption, the class of minimization-based principles exploiting the noise \emph{whiteness} property is one of the best performing \cite{hansenNCP,MF_whiteness,etna,SSVM_whiteness2021,JMIV_whiteness,BIT}. More specifically, one selects $\mu$ by minimizing the correlation between the residual image components, that is by guaranteeing that the residual image resembles as much as possible the underlying additive noise in terms of whiteness. The whiteness-based approaches have been proven to outperform the Morozov discrepancy principle in different imaging tasks, such as, e.g., denoising/restoration \cite{etna} and super-resolution \cite{SSVM_whiteness2021,JMIV_whiteness}. Nonetheless, despite the encouraging results on Gaussian data, so far the whiteness principle has not been extended to Poisson noise corruption. In this work, we are going to address such extension.

\subsection{Contribution}

The main contribution of this paper is to provide the first extension of the whiteness principle proposed for additive white noise to the case of Poisson noise. In particular, we will illustrate theoretically how our proposal simply relies on applying the standard whiteness principle to a suitably standardized version of the Poisson-corrupted observation and that it can be used to select the regularization parameter $\mu$ in any variational model of the $\mathcal{R}$-KL class. 

In this work, we apply the proposed selection strategy to the very popular TV-KL model
\begin{equation}
\label{eq:TVKL}
\bm{x}^*(\mu)\;{\in}\;\argmin_{\bm{x}\,\in\,\R_+^n}\{\,\mathrm{TV}(\bm{x}) \;{+}\; \mu\, \mathrm{KL}(\bm{\lambda};\bm{y})\,\}\,,\quad \bm{\lambda}=\bm{g}(\bm{\mathrm{H}x})+\bm{b}\,,
\end{equation}
employed for the image restoration (IR) and X-rays CT image reconstruction (CTIR) tasks. In the two application scenarios, the matrix $\bm{\mathrm{H}} \in \R^{m \times n}$ and the vectorial function $\bm{g}(\bm{h}) = \left(g(h_1),\ldots,g(h_m)\right)^T$ in \eqref{eq:TVKL} are specified as follows: 
\begin{equation}
\begin{array}{rll}
\mathrm{IR:}\;&\bm{\mathrm{H}} \in \R^{n\times n} \;\;\, \text{blurring matrix},   & g(h_i) = h_i, 
\\
\mathrm{CTIR:}\;&\bm{\mathrm{H}} \in \R^{m\times n} \;\: \text{Radon matrix},   & g(h_i) = I_0\,e^{-h_i}\!,\;\: I_0\in\R_{++}\,.
\end{array}
\label{eq:IRCTIR}
\end{equation}

From the numerical optimization viewpoint, in both scenarios the TV-KL model \eqref{eq:TVKL} will be solved by means of a two-blocks ADMM approach, which for the CTIR problem will be adopted in a semi-linearized version so as to significantly decrease the per-iteration computational cost. 

Experimental tests will show that in most cases the proposed selection principle outperforms the aforementioned ADP and NEDP and returns output images characterized by quality measures which are close to the ones achievable by manually tuning the regularization parameter $\mu$.

\medskip

The paper is organized as follows. In Section \ref{sec:prel} we set the notations and recall some preliminary results on white random processes. The proposed selection strategy is introduced in Section \ref{sec:PWP}, while in Section \ref{sec:admm} we outline the numerical scheme employed for the solution of model \eqref{eq:TVKL}. In Section \ref{sec:test} we extensively test the newly introduced strategy on IR and CTIR problems. Finally, we draw some conclusions and  provide an outlook for future research in Section \ref{sec:concl}.

\section{Notations and Preliminaries}
\label{sec:prel}
In this paper, scalars, vectors and matrices are denoted, e.g., by $x$, $\bm{x} = \left\{x_i\right\}$ and $\bm{\mathrm{X}} = \left\{x_{i,j}\right\}$ , respectively, whereas scalar random variables and random matrices (also referred to as random fields) are indicated by $X$ and $\bm{\mathcal{X}} = \left\{X_{i,j}\right\}$, respectively.
We indicate by $\mathrm{E}[X]$, $\mathrm{Var}[X]$, $\mathrm{Corr}[X,Y] = \mathrm{E}[XY]$ and $\mathrm{P}_X$ the expected value (or mean) of random variable $X$, the variance of $X$, the correlation between random variables $X$ and $Y$ and the probability mass function of (discrete) random variable $X$, respectively.
We denote by $\bm{0}_d$, $\bm{\I}_d$ and $\iota_{\mathrm{S}}$ the $d$-dimensional null (column) vector, the identity matrix of order $d$ and the indicator function of set ${\mathrm{S}}$, respectively, with $\iota_{\mathrm{S}}(\bm{x}) = 0$ for $\bm{x} \in {\mathrm{S}}$ and $\iota_{\mathrm{S}}(\bm{x}) = +\infty$ for $\bm{x} \notin {\mathrm{S}}$. We indicate respectively by $(\,\cdot\,,\,\cdot\,)$ and $(\,\cdot\,;\,\cdot\,)$ the concatenation by rows and by columns of scalars, vectors and matrices. Finally, $\| \,\cdot\,\|_2$ denotes the vector Euclidean norm or the matrix Frobenius norm, depending on the context

In order to introduce the theory underlying our proposal, it is useful to rewrite the vectorized image formation model  \eqref{eq:formod} in its equivalent matrix form. 
Denoting by $\,\bm{\mathrm{Y}}, \overline{\bm{\Lambda}} \in \R^{m_1 \times m_2}$ and $\bm{\mathrm{B}},\overline{\mathrm{\bm{X}}} \in \R^{n_1 \times n_2}$ the matrix forms of vectors $\,\bm{y},\overline{\bm{\lambda}}\in\R^m$ and $\bm{b},\overline{\bm{x}}\in\R^n$, respectively, it reads
\begin{equation}
\bm{\mathrm{Y}} \;\,{=}\,\; \bm{\mathrm{POISS}}
\left(\overline{\bm{\Lambda}}\right)\,, \quad\;
\overline{\bm{\Lambda}} = \bm{\mathrm{G}}\left(\bm{\mathrm{H}}\left(\overline{\bm{X}}\right)\right) + \bm{\mathrm{B}}\,,
\label{eq:formodmat}
\end{equation}
where, with a little abuse of notation, $\bm{\mathrm{H}}: \R^{n_1 \times n_1} \to \R^{m_1 \times m_1}$ indicates here the linear operator encoded by matrix $\bm{\mathrm{H}} \in \R^{m \times n}$ in the vectorized model \eqref{eq:formod}, and where $\bm{\mathrm{POISS}}\left(\,\overline{\bm{\mathrm{\Lambda}}}\,\right) = \left\{\mathrm{poiss}\left(\,\overline{\lambda}_{i,j}\right)\right\}$ and $\bm{\mathrm{G}}\left(\bm{\mathrm{H}}\left(\overline{\bm{X}}\right)\right)= \left\{g\left(\left(\bm{\mathrm{H}}\left(\overline{\bm{X}}\right)\right)_{i,j}\right)\right\}$, i.e. the matrix forms of vectors $\bm{\mathrm{poiss}}
\left(\,\overline{\bm{\lambda}}\,\right)$ and $\bm{g}(\bm{\mathrm{H} }\bar{\bm{x}})$ in \eqref{eq:formod}.

We now recall the definitions of weak stationary random field, ensemble normalized auto-correlation, sample normalized auto-correlation and, of particular importance for our purposes, white random field. To shorten notations in the definitions, we preliminarily define the following two sets of integer index pairs 
\begin{equation}
\begin{array}{lcl}
\mathrm{I} & \!\!\!{:=}\!\!\!\! & \left\{(i,j)\:\, \in \mathbb{Z}^2\!: \;(i,j) \:\,\in [1,m_1] \times [1,m_2] \,\right\},
\\
\mathrm{L} & \!\!\!{:=}\!\!\!\! & \left\{(l,m) \in \mathbb{Z}^2\!: \; (l,m) \in \left[-(m_1-1),(m_1-1)\right] \times \left[-(m_2-1),(m_2-1)\right] \,\right\}.
\end{array}
\end{equation}

\smallskip

\begin{definition}[weak stationary random field]
	\label{def:WSRF}
	A $m_1 \times m_2$ random field $\boldsymbol{\mathcal{Z}} = \left\{ Z_{i,j}\right\}$, $(i,j) \in \mathrm{I}$, is said to be weak stationary if
	\begin{equation}
	\begin{array}{l}
	\bullet\;\:\mathrm{E}\left[Z_{i,j}\right] \;{=}\; \mu_{\bm{\mathcal{Z}}} \in \R \,, \;\; \mathrm{Var}\left[Z_{i,j}\right] \;{=}\; \sigma^2_{\bm{\mathcal{Z}}} \in \R_{++}, \;\; \forall\,(i,j) \in \mathrm{I}\,;
	\\
	\bullet\;\:\mathrm{Corr}\left[Z_{i_1,j_1},Z_{i_1+l,j_1+m}\right] \;{=}\; \mathrm{Corr}\left[Z_{i_2,j_2},Z_{i_2+l,j_2+m}\right] \, ,  
	\\
	\;\;\;\;\forall \, (i_1,j_1) \in \mathrm{I}\,, \; \forall \,(i_2,j_2) \in \mathrm{I}\,,\;\forall \, (l,m) \in \mathrm{L}: \: (i_2+l,j_2+m) \in \mathrm{I} \, .
	\end{array}
	\nonumber
	\end{equation}
\end{definition}

\begin{definition}[ensemble normalized auto-correlation]
	\label{def:ENAC}
	The ensemble normalized auto-correlation of a $m_1 \times m_2$ weak stationary random field $\boldsymbol{\mathcal{Z}} = \left\{ Z_{i,j}\right\}$, $(i,j) \in \mathrm{I}$, is a $(2 m_1-1)\times(2 m_2 -1)$ matrix $\bm{\mathrm{A}}\left[\bm{\mathcal{Z}}\right]\;{=}\;\left\{a_{l,m}\left[\bm{\mathcal{Z}}\right]\right\}$, $(l,m) \in \mathrm{L}$, defined by
	\begin{equation}
	a_{l,m}\left[\bm{\mathcal{Z}}\right]\;{=}\; \frac{\mathrm{Corr}\left[Z_{i,j},Z_{i+l,j+m}\right]}{\sigma^2_{\bm{\mathcal{Z}}}}\,, \quad (l,m) \in \mathrm{L}\,, \; (i,j) \in \mathrm{I}: \: (i+l,j+m) \in \mathrm{I} \, .
	\end{equation}
	%
	%
\end{definition}

\begin{definition}[white random field]
	\label{def:WRF}
	A $m_1 \times m_2$ random field $\boldsymbol{\mathcal{Z}} = \left\{ Z_{i,j}\right\}$, $(i,j) \in \mathrm{I}$, is said to be white if
	\begin{equation}
	\begin{array}{l}
	\bullet\;\:\text{it is weak stationary with}\;\;\mu_{\bm{\mathcal{Z}}}\;{=}\; 0 \, ; \\
	\bullet\;\:\text{it is uncorrelated, that is its ensemble normalized autocorrelation}
	\\
	\;\;\;\:\bm{\mathrm{A}}\left[\bm{\mathcal{Z}}\right]\;{=}\;\left\{a_{l,m}\left[\bm{\mathcal{Z}}\right]\right\}, \;(l,m) \in \mathrm{L}, \;\text{satisfies:}
	\vspace{0.1cm}\\ 
	\qquad\qquad\qquad a_{l,m}\left[\bm{\mathcal{Z}}\right]\;{=}\;
	\left\{
	\begin{array}{ll}
	0     & \forall \, (l,m) \in \mathrm{L} \setminus \left\{(0,0)\right\} \, , \\
	1     & \text{if}\;\;(l,m)\;{=}\;(0,0) \, .
	\end{array}
	
	\right.
	\end{array}
	\nonumber
	\end{equation}
\end{definition}

\begin{definition}[sample normalized auto-correlation]
	\label{def:SNAC}
	The sample normalized auto-correlation of a $m_1 \times m_2$ non-zero matrix $\bm{\mathrm{Z}} = \left\{z_{i,j}\right\}$, $(i,j) \in \mathrm{I}$, is a $(2 m_1-1)\times(2 m_2 -1)$ matrix $\bm{\mathrm{S}}\left(\bm{\mathrm{Z}}\right)\;{=}\;\left\{s_{l,m}\left(\bm{\mathrm{Z}}\right)\right\}$, $(l,m) \in \mathrm{L}$, defined by
	\begin{equation}
	s_{l,m}\left(\bm{\mathrm{Z}}\right)\;{=}\; \frac{1}{\left\|\bm{\mathrm{Z}}\right\|_2^2}  \, \sum_{\;(i,j)\in\,\mathrm{I}} \!
	z_{i,j} \, z_{i+l,j+m} \, . 
	\label{eq:NAC_SP}
	\end{equation}
\end{definition}

It follows from Definition \ref{def:SNAC} that, given a non-zero matrix $\bm{\mathrm{Z}} = \left\{z_{i,j}\right\}$, $(i,j) \in \mathrm{I}$, one can measure the global amount of normalized auto-correlation between the entries of $\bm{\mathrm{Z}}$, that is how far is $\bm{\mathrm{Z}}$ from being the realization of a white random field, via the following scalar whiteness measure (\cite{etna,MF_whiteness}):
\begin{equation}
\mathcal{W}(\bm{\mathrm{Z}}) \;{:=}\; \left\|\bm{\mathrm{S}}\left(\bm{\mathrm{Z}}\right)\right\|_2^2\;{=}\;\sum_{(l,m)\in\,\mathrm{L}} \!\!\left(s_{l,m}\left(\bm{\mathrm{Z}}\right)\right)^2 \, ,
\label{eq:W}
\end{equation}
with scalars $s_{l,m}\left(\bm{\mathrm{Z}}\right)$ defined in \eqref{eq:NAC_SP}.

\section{The proposed whiteness principle for Poisson noise}
\label{sec:PWP}
In this section, we show how the residual whiteness principle proposed for additive white noise can be quite easily extended to the case of Poisson noise based on suitable random variable standardizations.

For this purpose, first in Definition \ref{def:PRVF} we recall the formal definition of Poisson random variable and Poisson independent random field, then in Definition \ref{def:SPRVF} we introduce their standard(ized) versions, whose main properties are finally highlighted in Proposition \ref{prop:SPRVF}.
\begin{definition}[Poisson random variable and independent random field]
	\label{def:PRVF}
	A discrete random variable $Y$ is said to be Poisson distributed with parameter $\lambda \:{\in}\: \R_{++}$, denoted by $Y \:{\sim}\: \mathcal{P}(\lambda)$, if its probability mass function reads
	\begin{equation}
	\mathrm{P}_Y(y\mid\lambda) \,\;{=}\; \frac{\lambda^y e^{-\lambda}}{y\,!} \, , \quad y \in \mathbb{N}\,.
	\label{eq:Y_PMF}
	\end{equation}
	The expected value and variance of random variable $Y$ are given by 
	\begin{equation}
	\mathrm{E}\left[Y\right] \,\;{=}\;\, \mathrm{Var}\left[Y\right] \,\;{=}\;\,\lambda \, .
	\end{equation} 
	A random field $\boldsymbol{\mathcal{Y}} = \left\{Y_{i,j}\right\}$ is said to be independent Poisson distributed with parameter $\bm{\Lambda} = \left\{\lambda_{i,j}\right\}$, denoted by $\boldsymbol{\mathcal{Y}} \sim \boldsymbol{\mathcal{P}}(\bm{\Lambda})$, if it satisfies:
	\begin{equation}
	Y_{i,j} \;{\sim}\; \mathcal{P}\left(\lambda_{i,j}\right) \;\:\forall \, (i,j) \in \mathrm{I}\,, \quad
	\mathrm{P}_{\boldsymbol{\mathcal{Y}}}(\bm{\mathrm{Y}}\mid \bm{\Lambda}) \;{=}\; \prod_{(i,j)\in\mathrm{I}} \mathrm{P}_{Y_{i,j}}(y_{i,j}\mid\lambda_{i,j}) \, .
	\label{eq:MMM}
	\end{equation}
\end{definition}

\begin{definition}[standard Poisson random variable and independent random field]
	\label{def:SPRVF}
	Let $Y \sim \mathcal{P}(\lambda)$. We call the discrete random variable $Z$ defined by
	\begin{equation}
	Z \,\;{=}\;\, S_{\lambda}(Y) \:\;{:=}\;\: \frac{Y -\mathrm{E}\left[Y\right]}{\sqrt{\mathrm{Var}\left[Y\right]}} \:\;{=}\;\: \frac{Y-\lambda}{\sqrt{\lambda}} \:\;{=}\;\: \frac{1}{\sqrt{\lambda}} \, Y - \sqrt{\lambda} \, ,
	\label{eq:S}
	\end{equation}
	as standard Poisson distributed with parameter $\lambda$, denoted by $\,Z \:{\sim}\: \widetilde{\mathcal{P}}(\lambda)$.
	
	Let $\boldsymbol{\mathcal{Y}} \sim \boldsymbol{\mathcal{P}}(\bm{\Lambda})$. We call the random field defined by 
	\begin{equation}
	\boldsymbol{\mathcal{Z}} \,\;{=}\; \left\{ Z_{i,j}\right\} \,\;\;\;\mathrm{with}\;\;\;\, Z_{i,j} \,\;{=}\;\, S_{\lambda_{i,j}}(Y_{i,j}) \:\;\; \forall \, (i,j) \in \mathrm{I} \, ,
	\label{eq:UUU}
	\end{equation}
	as independent standard Poisson distributed with parameter $\bm{\Lambda}$, denoted by $\boldsymbol{\mathcal{Z}} \sim \widetilde{\boldsymbol{\mathcal{P}}}(\bm{\Lambda})$.
\end{definition}
\begin{proposition}
	\label{prop:SPRVF}
	Let $Z \sim \widetilde{\mathcal{P}}(\lambda)$ and let $S_{\lambda}$ be the standardization function defined in \eqref{eq:S}. Then, the probability mass function, expected value and variance of random variable $Z$ are given by:
	\begin{eqnarray}
	&\mathrm{P}_{Z}\,(z|\lambda) \;{=}\; \displaystyle{\frac{\lambda^{S_{\lambda}^{-1}(z)} \, e^{-\lambda}}{\left(S_{\lambda}^{-1}(z)\right)\,!}}, \;\; z \in \left\{ S_{\lambda}(0),S_{\lambda}(1),\ldots\right\}, \;\;S_{\lambda}^{-1}(z) \,\;{=}\;\, \sqrt{\lambda}\,z+\lambda,&
	\label{eq:Ybar_PMF1}
	\\
	&\mathrm{E}\left[\,Z\,\right] \;{=}\; 0\,, \quad 
	\mathrm{Var}\left[\,Z\,\right] \;{=}\; 1 \, .&
	\label{eq:Ybar_PMF2}
	\end{eqnarray}
	Hence, any independent standard Poisson random field $\boldsymbol{\mathcal{Z}} \sim \widetilde{\boldsymbol{\mathcal{P}}}(\bm{\Lambda})$ is white.
\end{proposition}
\begin{proof}
	The scalar affine standardization function $S_{\lambda}: \mathbb{N} \to \{S_{\lambda}(0),S_{\lambda}(1),\ldots\}$ in \eqref{eq:S} is bijective (as $\lambda \in \R_{++}$), hence it admits the inverse $S_{\lambda}^{-1}$ defined in \eqref{eq:Ybar_PMF1}. The expression of $\mathrm{P}_Z$ in \eqref{eq:Ybar_PMF1} thus comes from specifying the general form of the probability mass function of a discrete random variable defined by a bijective function of another discrete random variable. The fact that $Z$ has zero-mean and unit-variance - as stated in \eqref{eq:Ybar_PMF2} - comes immediately from the definition of $S_{\lambda}$ in \eqref{eq:S}.
	
	It thus follows from the definition of a standard Poisson independent random field $\boldsymbol{\mathcal{Z}} \,\;{=}\; \left\{ Z_{i,j}\right\}$ given in \eqref{eq:UUU} and from statement \eqref{eq:Ybar_PMF2} that:
	\begin{equation}
	Z_{i,j} \sim \widetilde{\mathcal{P}}\left(\lambda_{i,j}\right) \,\;{\Longrightarrow}\;\,
	\left\{\!\!
	\begin{array}{rcl}
	\mathrm{E}\left[Z_{i,j}\right] &\!\!\!\!{=}\!\!\!\!& \mu_{\bm{\mathcal{Z}}} = 0 \\
	\mathrm{Var}\left[Z_{i,j}\right]     &\!\!\!\!{=}\!\!\!\!& \sigma^2_{\bm{\mathcal{Z}}} = 1
	\end{array}
	\right. , \;
	\forall\,(i,j) \, .
	\label{eq:OOO}
	\end{equation}
	
	Moreover, it clearly comes from independence of a non-standard Poisson random field $\bm{\mathcal{Y}}$ - formalized in \eqref{eq:MMM} - and from the entry-wise definition of random field standardization in \eqref{eq:UUU} that independence also holds true for a standard Poisson random field $\boldsymbol{\mathcal{Z}}$; in formula:
	\begin{equation}
	\mathrm{P}_{\boldsymbol{\mathcal{Z}}}(\bm{\mathrm{Z}}\mid \bm{\Lambda}) \;{=}\; \prod_{(i,j)} \mathrm{P}_{Z_{i,j}}(z_{i,j}\mid\lambda_{i,j}) \, .
	\end{equation}
	Since independence implies uncorrelation and based on \eqref{eq:OOO}, we have
	\begin{equation}
	\mathrm{Corr}\left[Z_{i_1,j_1},Z_{i_2,j_2}\right] \;{=}\; \left\{
	\begin{array}{ll}
	0 & \text{for}\;\: (i_1,j_1) \neq (i_2,j_2)\,, 
	\\
	\mathrm{Var}\left[ Z_{i_1,j_1}\right] \;{=}\; \sigma^2_{\bm{\mathcal{Z}}} \;{=}\; 1 & \text{for}\;\: (i_1,j_1) = (i_2,j_2) \,.  
	\end{array}
	\right.
	\label{eq:CCC}
	\end{equation}
	It follows from \eqref{eq:OOO}, \eqref{eq:CCC} and from Definition \ref{def:WSRF} that $\bm{\mathcal{Z}}$ is a weak stationary random field. Then, it comes from \eqref{eq:CCC} and from Definition \ref{def:ENAC} that the ensemble normalized auto-correlation $\bm{\mathrm{A}}\left[\bm{\mathcal{Z}}\right]\;{=}\;\left\{a_{l,m}\left[\bm{\mathcal{Z}}\right]\right\}$ satisfies
	\begin{equation}
	a_{l,m}\left[\bm{\mathcal{Z}}\right]\;{=}\;
	\left\{
	\begin{array}{ll}
	0     & \text{for}\;\:(l,m) \neq (0,0) \, , \\
	1     & \text{for}\;\:(l,m) = (0,0) \, .
	\end{array}
	\right.
	\label{eq:EEE}
	\end{equation}
	Hence, based on Definition \ref{def:WRF}, we can conclude that any standard Poisson independent random field $\bm{\mathcal{Z}}$ is white.
\end{proof}


In light of Definition \ref{def:PRVF}, the image formation model \eqref{eq:formodmat} can be written in probabilistic terms as follows:
\begin{equation}
\label{eq:formodmatp}
\bm{\mathrm{Y}}\;\;\text{realization of}\;\;
\boldsymbol{\mathcal{Y}} \sim \boldsymbol{\mathcal{P}}\left(\,\overline{\boldsymbol{\Lambda}}\,\right)\,,
\end{equation}
with matrix $\overline{\bm{\Lambda}}$ defined in \eqref{eq:formodmat}.

Then, based on Definition \ref{def:SPRVF}, after introducing the matrix
\begin{equation}
\label{eq:Z}
\bm{\mathrm{Z}} \,\;{=}\;\, \left\{z_{i,j}\right\} \;\;\mathrm{with}\;\; z_{i,j}\,\;{=}\;\,S_{\,\overline{\lambda}_{i,j}}\!\left(y_{i,j}\right)\,\;{=}\;\,\frac{y_{i,j}-\overline{\lambda}_{i,j}}{\sqrt{\,\overline{\lambda}_{i,j}}} \, ,  
\end{equation}
the probabilistic model \eqref{eq:formodmatp} can be equivalently written in standardized form as
\begin{equation}
\label{eq:formodmatpn}
\bm{\mathrm{Z}}\;\;\text{realization of}\;\;
\boldsymbol{\mathcal{Z}} \sim \widetilde{\boldsymbol{\mathcal{P}}}\left(\,\overline{\boldsymbol{\Lambda}}\,\right) \,.
\end{equation}
That is, matrix $\bm{\mathrm{Z}}$  in \eqref{eq:Z} with $\overline{\bm{\Lambda}}$ in \eqref{eq:formodmat} is the realization of an independent standard Poisson random field $\boldsymbol{\mathcal{Z}}$ which, according to Proposition \ref{prop:SPRVF}, is white.

We note that $\bm{\mathrm{Z}}$ can not be computed in practice as it depends on $\overline{\bm{\Lambda}}$ which, in its turn, depends on the unknown uncorrupted image $\overline{\bm{\mathrm{X}}}$.
However, the whiteness property of $\bm{\mathrm{Z}}$ can be exploited for stating a new principle for automatically selecting the value of the regularization parameter $\mu$ in the class of $\mathcal{R}$-KL variational models.

Denoting by $\,\bm{\mathrm{X}}^*(\mu) = \left\{x_{i,j}^*(\mu)\right\}\,$ the matrix form of the solution of a $\mathcal{R}$-KL model - e.g., of the TV-KL model in \eqref{eq:TVKL} - we introduce the $\mu$-dependent matrices $\bm{\Lambda}^*(\mu), \bm{\mathrm{Z}}^*(\mu) \in \R^{m_1 \times m_2}$ given by
\begin{eqnarray}
\bm{\Lambda}^*(\mu) &\!\!\!{=}\!\!\!&
\left\{ \lambda_{i,j}^*(\mu)\right\} \,\;{=}\;\, \bm{\mathrm{G}}\left(\bm{\mathrm{H}}\left(\bm{\mathrm{X}}^*(\mu)\right)\right) + \bm{\mathrm{B}},
\label{eq:Lstar}
\\
\bm{\mathrm{Z}}^*(\mu) &\!\!\!{=}\!\!\!& \left\{ z_{i,j}^*(\mu)\right\} \,\;\;\mathrm{with}\;\;\, z_{i,j}^*(\mu) \,\;{=}\;\,S_{\,\lambda_{i,j}^*(\mu)}\!\left(y_{i,j}\right) \,\;{=}\;\, \frac{y_{i,j}-\lambda_{i,j}^*(\mu)}{\sqrt{\lambda_{i,j}^*(\mu)}},
\label{eq:Zstar}
\end{eqnarray}
The ideal goal of any criterion for choosing $\mu$ in the TV-KL model is to select the value $\mu^*$ yielding the closest solution image $\bm{\mathrm{X}}^*(\mu^*)$ to the target uncorrupted image $\overline{\bm{\mathrm{X}}}$, according to some distance metric.
The conjecture behind our proposal is that the closer the solution  $\bm{\mathrm{X}}^*(\mu)$ is to the target $\overline{\bm{\mathrm{X}}}$, the closer the matrix $\bm{\mathrm{Z}}^*(\mu)$ defined in (\ref{eq:Lstar})-(\ref{eq:Zstar}) will be to $\bm{\mathrm{Z}}$ in \eqref{eq:Z}, so the more $\bm{\mathrm{Z}}^*(\mu)$ will resemble the realization of a white random field. Hence, the proposed criterion, that we refer to as the Poisson Whiteness Principle (PWP), consists in choosing the value of $\mu$ which leads to the less auto-correlated matrix $\bm{\mathrm{Z}}^*(\mu)$. Based on the scalar auto-correlation measure introduced in \eqref{eq:W}, the proposed PWP reads:
\begin{equation}
\boxed{
	\begin{array}{c}
	\text{Select}\;\;\mu\;{=}\;\mu^* \:{\in}\; 
	\displaystyle{\argmin_{\mu \in \R_{++}} 
		\left\{\, W(\mu) \,\;{:=}\;\, \mathcal{W}\left(\bm{\mathrm{Z}}^*(\mu)\right)\,\right\}} \, , 
	\vspace{0.15cm}\\
	\text{with matrix $\bm{\mathrm{Z}}^*(\mu)$ defined in \eqref{eq:Lstar}-\eqref{eq:Zstar} and function $\mathcal{W}$ in \eqref{eq:W}.}
	\end{array}
}
\label{eq:PWP}
\tag{PWP}
\end{equation}

\section{Numerical solution by ADMM}
\label{sec:admm}

In this section, we address the numerical solution of the TV-KL model \eqref{eq:TVKL} for the IR and CTIR imaging problems, that is for matrix $\bm{\mathrm{H}}$ and function $\bm{g}$ defined as in \eqref{eq:IRCTIR}.

Recalling the definition of TV in \eqref{eq:TV} and introducing the discrete gradient matrix $\bm{\D} := (\bm{\D}_h;\bm{\D}_v) \in \R^{2n \times n}$ with $\bm{\D}_h,\bm{\D}_v \in \R^{n \times n}$ two finite difference matrices discretizing the first-order partial derivatives of image $\bm{x}$ in the horizontal and vertical direction, respectively, we write the TV-KL model \eqref{eq:TVKL} in the following form:
\begin{equation}
\label{eq:subuwunc}
\bm{x}^*\;{\in}\;\argmin_{\bm{x} \in \R^n}\left\{ \sum_{i=1}^n\|(\bm{\D x})_i\|_2 \;{+}\; \mu\,\mathrm{KL}\left(\bm{g}(\bm{\mathrm{H}x}\right)+\bm{b};\bm{y}) \;{+}\; \iota_{\R_+^n}(\bm{x})\right\}\, ,
\end{equation}
where, with a little abuse of notation,  $(\bm{\D x})_i := \left( \left(\bm{\D}_h \bm{x}\right)_i \, ; \, \left(\bm{\D}_v \bm{x}\right)_i \right) \in \R^2$, the discrete gradient of image $\bm{x}$ at pixel $i$.

By introducing the auxiliary variables $\bm{t}_1\;{=}\;\bm{\D x} \in \R^{2n}$, $\bm{t}_2\;{=}\;\bm{\mathrm{H} x} \in \R^m$ and $\bm{t}_3\;{=}\;\bm{x} \in \R^n$, problem~\eqref{eq:subuwunc} can be equivalently rewritten in the following linearly constrained form:
\begin{equation}
\begin{array}{lcl}
\left\{\bm{x}^*\!,\bm{t}_1^*,\bm{t}_2^*,\bm{t}_3^*\right\} & \!\!\!\!{\in}\!\!\!\! &  \displaystyle{\argmin_{\bm{x},\bm{t}_1,\bm{t}_2,\bm{t}_3}\left\{
	\sum_{i=1}^n\|\bm{t}_{1,i}\|_2 + \mu\,\mathrm{KL}\left(\bm{g}(\bm{t}_2)+\bm{b};\bm{y}\right)+\iota_{\R_+^n}(\bm{t}_3)
	\!\right\}}
\\
&  \!\!\!\!\!\!\!\! & 
\text{subject to:}\;\;
\bm{t}_1\;{=}\;\bm{\D x},\;\; \bm{t}_2\;{=}\;\bm{\mathrm{H} x},\;\;
\bm{t}_3\;{=}\;\bm{x},
\end{array}
\label{eq:TGVKLcon}
\end{equation}
where $\bm{t}_{1,i} := (\bm{\D x})_i \in \R^2$.

It is easy to prove - see, e.g., \cite{DLV_dtgv} - that, after introducing the total auxiliary variable $\,\bm{t}:=(\bm{t}_1;\bm{t}_2;\bm{t}_3) \in \R^{m+3n}$, problem~\eqref{eq:TGVKLcon} takes the form:
\begin{equation}
\begin{array}{lcl}
\left\{\bm{x}^*,\bm{t}^*\right\} & \!\!{\in}\!\! & \displaystyle{\argmin_{\bm{x},\bm{t}}\left\{\,C_1(\bm{x}) + C_2(\bm{t})\,\right\}}
\\
&\!\!\!\!& \text{subject to:}\;\;
\bm{\mathrm{M}}_1 \bm{x}+\bm{\mathrm{M}}_2\bm{t}=\bm{0},     
\end{array}
\label{eq:modelnew}
\end{equation}
where the two cost functions $C_1: \R^n \to \R$ and $C_2: \R^{m+3n} \to \R$ are defined by
\begin{equation}
C_1(\bm{x})\;{=}\;0, \quad
C_2(\bm{t})\;{=}\; \sum_{i=1}^n\|\bm{t}_{1,i}\|_2 + \mu\, \mathrm{KL}\left(\bm{g}(\bm{t}_2)+\bm{b};\bm{y}\right)+\iota_{\R_+^n}(\bm{t}_3)\,,
\label{eq:F2}
\end{equation}
and the two matrices $\bm{\mathrm{M}}_1 \in \R^{(m+3n) \times n}$ and $\bm{\mathrm{M}}_2 \in \R^{(m+3n) \times (m+3n)}$ read
\begin{equation}
\bm{\mathrm{M}}_1=\left( \bm{\mathrm{D}};\bm{\mathrm{H}};\bm{\I}_n\right)\,,\quad \bm{\mathrm{M}}_2=-\bm{\I}_{m+3n}\,.
\label{eq:AB}
\end{equation}

Functions $C_1$ and $C_2$ in \eqref{eq:F2} are both proper, lower semi-continuous and convex, hence problem (\ref{eq:modelnew})-(\ref{eq:AB}) is a standard  two-blocks separable optimization problem which can be solved by ADMM~\cite{admm}.

The augmented Lagrangian function associated to problem \eqref{eq:modelnew} reads
\begin{eqnarray}
\mathcal{L}(\bm{x},\bm{t},\bm{\rho};\beta) &\!\!\!\!{=}\!\!\!\!& C_1(\bm{x})+C_2(\bm{t})+\langle\bm{\rho},\bm{\mathrm{M}}_1\bm{x}+\bm{\mathrm{M}}_2\bm{t}\rangle+\frac{\beta}{2}\|\bm{\mathrm{M}}_1\bm{x}+\bm{\mathrm{M}}_2\bm{t}\|_2^2,
\label{eq:AL} 
\end{eqnarray}
where $\bm{\rho}\in\R^{m+3n}$ 
is the vector of Lagrange multipliers associated to the linear constraint in~\eqref{eq:modelnew} and $\beta\in\R_{++}$ is the ADMM penalty parameter.

Solving problem \eqref{eq:modelnew} amounts to seek the saddle point(s) of the augmented Lagrangian function which, according to the standard two-blocks ADMM, can be computed as the limit point of the following iterative procedure:
\begin{align}
\label{eq:x_ADMM}
\bm{x}^{(k+1)}\:\;{\in}\;\:&\argmin_{\bm{x}\,\in\,\R^{n}}\mathcal{L}(\bm{x},\bm{t}^{(k)},\bm{\rho}^{(k)};\beta)\\
\label{eq:z_ADMM}
\bm{t}^{(k+1)}\:\;{\in}\;\:&\argmin_{\bm{t}\,\in\,\R^{m+3n}}\mathcal{L}(\bm{x}^{(k+1)},\bm{t},\bm{\rho}^{(k)};\beta)\\
\label{eq:getzeta}
\bm{\rho}^{(k+1)}\:\;{=}\;\:&\bm{\rho}^{(k)}\;{+}\; \beta\left(\bm{\mathrm{M}}_1\bm{x}^{(k+1)}+\bm{\mathrm{M}}_2\bm{t}^{(k+1)}\right)\,.
\end{align}
In what follows, we will detail how to solve  \eqref{eq:x_ADMM}-\eqref{eq:z_ADMM} when tackling the IR and CTIR imaging problems.

\subsection{The \textbf{x}-subproblem}

Recalling the definition of the augmented Lagrangian function $\mathcal{L}$ in \eqref{eq:AL}, with functions $C_1,C_2$ in \eqref{eq:F2} and matrices $\bm{\mathrm{M}}_1, \bm{\mathrm{M}}_2$ in \eqref{eq:AB}, after dropping the constant terms the $\bm{x}$-update problem in~\eqref{eq:x_ADMM} reads
\begin{eqnarray}
\!\!\!\!\!\!\bm{x}^{(k+1)}&\!\!\!\!{\in}\!\!\!\!&
\argmin_{\bm{x}\,\in\,\R^{n}}\left\{\langle\bm{\rho}^{(k)},\bm{\mathrm{M}}_1\bm{x}-\bm{t}^{(k)}\rangle+\frac{\beta}{2}\|\bm{\mathrm{M}}_1\bm{x}-\bm{t}^{(k)}\|_2^2\right\}
\nonumber\\
&\!\!\!\!{=}\!\!\!\!&\argmin_{\bm{x}\,\in\,\R^{n}}\left\{Q^{(k)}(\bm{x}) := \frac{1}{2}\|\bm{\mathrm{M}}_1\bm{x}\;{-}\;\bm{v}^{(k)}\|_2^2\right\},\;\: \bm{v}^{(k)}\;{=}\;\bm{t}^{(k)}-\frac{1}{\beta}\,\bm{\rho}^{(k)}.
\label{eq:subx1}
\end{eqnarray}
Since the cost function $Q^{(k)}$ in \eqref{eq:subx1} is quadratic and convex, it admits global minimizers which are the solutions of the linear system of normal equations:
\begin{equation}
\bm{\mathrm{M}}_1^{\mathrm{T}}  \bm{\mathrm{M}}_1 \, \bm{x}^{(k+1)} \;{=}\; \bm{\mathrm{M}}_1^{\mathrm{T}} \bm{v}^{(k)}
\;\;{\Longleftrightarrow}\;\; \left(\bm{\mathrm{D}}^{\mathrm{T}}\bm{\mathrm{D}} + \bm{\mathrm{H}}^{\mathrm{T}}\bm{\mathrm{H}} + \bm{\mathrm{I}}_n \right) \bm{x}^{(k+1)} \;{=}\; \bm{\mathrm{M}}_1^{\mathrm{T}} \bm{v}^{(k)}.
\label{eq:subx2}
\end{equation}
The coefficient matrix in \eqref{eq:subx2} has full rank independently of matrices $\bm{\D}$ and $\bm{\mathrm{H}}$ - i.e., of the finite difference discretization used for the gradient and of the imaging application considered - hence the solution of \eqref{eq:subx1} is unique and reads
\begin{equation}
\bm{x}^{(k+1)} =  \left(\bm{\mathrm{M}}_1^{\mathrm{T}} \bm{\mathrm{M}}_1\right)^{-1} \bm{\mathrm{M}}_1^{\mathrm{T}} \bm{v}^{(k)} \, .
\end{equation} 
For the IR inverse problem, upon the assumption of space-invariant blur and periodic boundary conditions, the coefficient matrix in \eqref{eq:subx2} is block-circulant with circulant blocks. Hence, the above linear system can be solved very efficiently by one application of the 2D Fast Fourier Transform (FFT) and one application of the inverse 2D FFT.

When addressing the CTIR problem, the structure of matrix $\bm{\mathrm{H}}$ - which, we recall, in this case is a Radon matrix - does not allow for a Fourier diagonalization of matrix $\bm{\mathrm{M}}_1^{\mathrm{T}} \bm{\mathrm{M}}_1$, thus yielding a significative computational burden related to the solution of linear system \eqref{eq:subx2}. A popular strategy for avoiding such difficulty is the linearized ADMM. It relies on computing $\bm{x}^{(k+1)}$ as the global minimizer of a surrogate function $\widehat{Q}^{(k)}$ of $Q^{(k)}$ in \eqref{eq:subx1}, namely
\begin{equation}
\bm{x}^{(k+1)} \;{=}\; \argmin_{\bm{x}\,\in\,\R^{n}} \,\widehat{Q}^{(k)}(\bm{x}) \, ,
\label{eq:subx7}
\end{equation}
where $\widehat{Q}^{(k)}$ is a quadratic function of the following form
\begin{eqnarray}
\widehat{Q}^{(k)}(\bm{x})&{=}&Q^{(k)}(\bm{x}^{(k)}) + \langle \nabla Q^{(k)}(\bm{x}^{(k)}),\bm{x}-\bm{x}^{(k)}\rangle
\nonumber\\
&& + \frac{\eta}{2} \| \bm{x}-\bm{x}^{(k)} \|_2^2, \quad \eta \;{\geq}\; \| \bm{\mathrm{M}}_1 \|_2^2 \, .
\label{eq:subx8}
\end{eqnarray}
It can be easily proved that any function $\widehat{Q}^{(k)}$ in \eqref{eq:subx8} is a \emph{quadratic tangent majorant} of the original function $Q^{(k)}$ in \eqref{eq:subx1} at point $\bm{x}^{(k)}$, that is it satisfies
\begin{equation}
\begin{array}{c}
\widehat{Q}^{(k)}(\bm{x}^{(k)}) \;{=}\; Q^{(k)}(\bm{x}^{(k)}), \quad
\nabla \widehat{Q}^{(k)}(\bm{x}^{(k)}) \;{=}\; \nabla Q^{(k)}(\bm{x}^{(k)}), 
\\
\widehat{Q}^{(k)}(\bm{x}) \;{\geq}\; Q^{(k)}(\bm{x}) \;\forall\,\bm{x} \in \R^n \, .
\end{array}
\end{equation}
It follows from (\ref{eq:subx7})-(\ref{eq:subx8}) that the new iterate $\bm{x}^{(k+1)}$ computed by the linearized ADMM is given by
\begin{eqnarray}
\bm{x}^{(k+1)} &\!\!\!\!{=}\!\!\!\!&  \argmin_{\bm{x}\,\in\,\R^{n}} \left\{ \langle \nabla Q^{(k)}(\bm{x}^{(k)}),\bm{x}\rangle + \frac{\eta}{2} \| \bm{x}-\bm{x}^{(k)} \|_2^2\right\}
\label{eq:subx9}\\
&\!\!\!\!{=}\!\!\!\!& 
\bm{x}^{(k)} - \frac{1}{\eta} \nabla Q^{(k)}(\bm{x}^{(k)})
\label{eq:subx10}\\
&\!\!\!\!{=}\!\!\!\!& 
\bm{x}^{(k)} - \frac{1}{\eta} \bm{\mathrm{M}}_1^{\mathrm{T}} \left( \bm{\mathrm{M}}_1 \bm{x}^{(k)}-\bm{v}^{(k)}\right), \quad \eta \;{\geq}\; \| \bm{\mathrm{M}}_1 \|_2^2 \, ,
\label{eq:subx11}
\end{eqnarray}
where in \eqref{eq:subx9} we dropped the constant terms, in \eqref{eq:subx10} we set $\bm{x}^{(k+1)}$ equal to the unique stationary point of the strongly convex cost function in \eqref{eq:subx9} and, finally, in \eqref{eq:subx11} we substituted the explicit expression of the gradient of the original cost function $Q^{(k)}$ defined in \eqref{eq:subx1}.

\subsection{The \textbf{t}-subproblem}

Recalling definitions (\ref{eq:F2})-(\ref{eq:AL}), the $\bm{t}$-subproblem in~\eqref{eq:z_ADMM} reads
\begin{eqnarray}
\bm{t}^{(k+1)}&\!\!\!\!{\in}\!\!\!\!&
\argmin_{\bm{t}\,\in\,\R^{m+3n}}\left\{C_2(\bm{t})+\langle\bm{\rho}^{(k)},\bm{\mathrm{M}}_1\bm{x}^{(k+1)}-\bm{t}\rangle+\frac{\beta}{2}\|\bm{\mathrm{M}}_1\bm{x}^{(k+1)}-\bm{t}\|_2^2\right\}
\nonumber\\
&\!\!\!\!{=}\!\!\!\!&\argmin_{\bm{t}\,\in\,\R^{m+3n}}\left\{C_2(\bm{t})+ \frac{\beta}{2}\|\bm{t}-\bm{q}^{(k)}\|_2^2\right\},\;\; \bm{q}^{(k)}\;{=}\;\bm{\mathrm{M}}_1\bm{x}^{(k+1)}+\frac{1}{\beta}\,\bm{\rho}^{(k)}\!.
\label{eq:subz1}
\end{eqnarray}
Then, by recalling the definition of function $C_2$ in \eqref{eq:F2} and introducing the vectors $\bm{\rho}_1^{(k)} \in \R^{2n}$, $\bm{\rho}_2^{(k)} \in \R^m$ and $\bm{\rho}_3^{(k)} \in \R^n$ such that $\bm{\rho}^{(k)} = \big(\bm{\rho}_1^{(k)};\bm{\rho}_2^{(k)};\bm{\rho}_3^{(k)}\big)$ and the vectors 
\begin{equation}
\begin{array}{c}
\displaystyle{\bm{q}_1^{(k)} =  \bm{\mathrm{D}}\bm{x}^{(k+1)}+\frac{1}{\beta}\bm{\rho}_1^{(k)} \in \R^{2n}, 
	\quad\;
	\bm{q}_2^{(k)} = \bm{\mathrm{H}}\bm{x}^{(k+1)}+\frac{1}{\beta}\bm{\rho}_2^{(k)} \in \R^m,}
\vspace{0.15cm}\\
\displaystyle{\bm{q}_3^{(k)} =  \bm{x}^{(k+1)}+\frac{1}{\beta}\bm{\rho}_3^{(k)} \in \R^n,}
\end{array}
\end{equation}
such that $\bm{q}^{(k)} = \big(\bm{q}_1^{(k)};\bm{q}_2^{(k)};\bm{q}_3^{(k)}\big)$, problem \eqref{eq:subz1} can be equivalently written as
\begin{equation}
\bm{t}^{(k+1)} \;{\in}\;\argmin_{\bm{t}\,\in\,\R^{m+3n}} \left\{\,T_1\left(\bm{t}_1\right) \;{+}\; T_2\left(\bm{t}_2\right) \;{+}\; T_3\left(\bm{t}_3\right) \,\right\} \, , \;\;\text{with:}
\end{equation}
\begin{equation}
\begin{array}{rcrl}
\displaystyle{T_1\left(\bm{t}_1\right)} &\!\!{=}\!\!\!& \displaystyle{\sum_{i=1}^n\|\bm{t}_{1,i}\|_2} &\!\!\!{+}\;\: \displaystyle{\frac{\beta}{2}\,\|\bm{t}_1-\bm{q}_1^{(k)}\|_2^2\,,} 
\vspace{0.0cm}\\
\displaystyle{T_2\left(\bm{t}_2\right)} &\!\!{=}\!\!\!& \displaystyle{\mu\, \mathrm{KL}\left(\bm{g}(\bm{t}_2)+\bm{b};\bm{y}\right)}&\!\!\!{+}\;\: \displaystyle{\frac{\beta}{2}\,\|\bm{t}_2-\bm{q}_2^{(k)}\|_2^2 \, ,}
\vspace{0.15cm}\\
\displaystyle{T_3\left(\bm{t}_3\right)} &\!\!{=}\!\!\!& \displaystyle{\iota_{\R_+^n}(\bm{t}_3)}&\!\!\!{+}\;\: \displaystyle{\frac{\beta}{2}\,\|\bm{t}_3-\bm{q}_3^{(k)}\|_2^2 \, .}
\vspace{0.2cm}
\end{array}
\label{eq:subz2}
\end{equation}
Therefore, the updates of variables $\bm{t}_1$, $\bm{t}_2$ and $\bm{t}_3$ can be addressed separately.

\vspace{0.2cm}\emph{Update of} $\bm{t}_1$. It comes from \eqref{eq:subz2} that the update of $\bm{t}_1$ reads
\begin{equation}
\bm{t}_1^{(k+1)} \;{=}\; \argmin_{\bm{t}_1 \in \R^{2n}}
\left\{ \sum_{i=1}^n 
\left[ \left\|\bm{t}_{1,i}\right\|_2 + 
\frac{\beta}{2}\left( 
\bm{t}_{1,i}-\bm{q}_{1,i}^{(k)}
\right)^2
\right]
\right\}\,.
\label{eq:z2_1}
\end{equation}
Hence, problem \eqref{eq:z2_1} is  separable into $n$ independent $2$-dimensional problems
\begin{equation}
\bm{t}_{1,i}^{(k+1)} 
\,\;{=}\;\,
\argmin_{\bm{t}_{1,i} \,\in\, \R^{2n}}
\left\{ \left\|\bm{t}_{1,i}\right\|_2 + 
\frac{\beta}{2}\left( 
\bm{t}_{1,i}-\bm{q}_{1,i}^{(k)}
\right)^2
\right\} , \quad i = 1,\ldots,n \, ,
\label{eq:LKM}
\end{equation}
which represent the proximal map of the Euclidean norm function $\|\,\cdot\,\|_2$ in $\R^2$ calculated at points $\bm{q}_{1,i}^{(k)}$, $i = 1,\ldots,n$. Such a proximal map admits a well-known explicit expression which leads to the following closed-form solution of problem \eqref{eq:LKM}:
\begin{equation}
\bm{t}_{1,i}^{(k+1)} 
\,\;{=}\;\, 
\max\left\{\,
\left\| \bm{q}_{1,i}^{(k)} \right\|_2 - \frac{1}{\beta} \,\, , \, 0 \, \right\} \, \frac{\bm{q}_{1,i}^{(k)}}{\left\|\bm{q}_{1,i}^{(k)}\right\|_2} , \quad i = 1,\ldots,n \, .
\end{equation}
where $\, 0 \, \cdot \, \bm{0} \,/\,0 \;{=}\; \bm{0}\,$ is assumed.


\vspace{0.2cm}\emph{Update of} $\bm{t}_2$. 
It follows from \eqref{eq:subz2} that, after introducing the scalar $\tau = \mu /\beta$, the updated vector $\bm{t}_2^{(k+1)}$ is given by
\begin{eqnarray}
\bm{t}_2^{(k+1)}&\!\!\!\!{\in}\!\!\!\!& \argmin_{\bm{t}_2\,\in\,\R^m}\left\{\tau\,\mathrm{KL}(\bm{g}(\bm{t}_2)+\bm{b};\bm{y})\;{+}\;\frac{1}{2}\|\bm{t}_2-\bm{q}_2^{(k)}\|_2^2\right\}
\nonumber\\
&\!\!\!\!{=}\!\!\!\!& \argmin_{\bm{t}_2\,\in\,\R^m}\left\{ \sum_{i=1}^m \left[ \tau\, g(t_i) - \tau \, y_i \ln\left(g(t_i)+b_i\right)+\frac{1}{2}\left(t_i-q_i\right)^2\right]\right\},
\label{eq:z1upd}
\end{eqnarray}
where in \eqref{eq:z1upd} we substituted the explicit expression of the KL divergence term reported in \eqref{eq:KL}, we dropped the constants and, for simplicity of notation, we set $t_i:=t_{2,i} \in \R$ and $q_i = q_{2,i}^{(k)} \in \R$. Hence, similarly to the $\bm{t}_1$ update problem in \eqref{eq:z2_1}, the $m$-dimensional minimization problem \eqref{eq:z1upd} is equivalent to the $m$ following $1$-dimensional problems 
\begin{equation}
\label{eq:zupd}
t_i^{(k+1)}\;{=}\;\argmin_{t_i\,\in\,\R}\left\{\tau \,g(t_i) - \tau \, y_i \ln\left(g(t_i)+b_i\right)+\frac{1}{2}\left(t_i-q_i\right)^2\right\},
\end{equation}
$i=1,\ldots,m$.

In the IR scenario, i.e. when $g(t_i)=t_i$, the cost function in \eqref{eq:zupd} is infinitely many times differentiable, strictly convex and coercive in its domain $t_i \in (-b_i,+\infty)$. Hence, the solution $t_i^{(k+1)}$ of \eqref{eq:zupd} exists, is unique and coincides with the unique stationary point of the cost function, given by
\begin{equation}\label{eq:getz}
t_i^{(k+1)} = \frac{1}{2}\left[-(\tau+b_i-q_i)+\sqrt{(\tau+b_i-q_i)^2+4\left(q_i\,b_i+\tau(y_i-b_i)\right)
}\,\right]\,.
\end{equation}

For the CTIR problem, i.e. when $g(t_i) = I_0 e^{-t_i}$, problem \eqref{eq:zupd} reads
\begin{equation}
\label{eq:zupdct}
t_i^{(k+1)}\;{=}\;\argmin_{t_i\,\in\,\R}\left\{\tau\,I_0\, e^{-t_i}-\tau \, y_i \ln\left(I_0\, e^{-t_i}\!+b_i\right)  +\frac{1}{2}(t_i-q_i)^2
\right\}.
\end{equation}
The cost function in \eqref{eq:zupdct} is infinitely many times differentiable and coercive in its domain $t_i \in \R$, hence it admits global minimizers. However, in the general case of a nonzero background, i.e. when $b_i \in \R_{++}$, problem \eqref{eq:zupdct} does not admit a closed-form solution and can only be addressed by employing iterative solvers.

On the other hand, when $b_i=0$ the cost function is also strictly convex, hence $t_i^{(k+1)}$ in \eqref{eq:zupdct} is given by the unique solution of the first-order optimatily condition 
\begin{equation}
-\tau \,I_0\, e^{-t_i} + \tau\, y_i  + t_i-q_i = 0\,.   
\end{equation}
The above nonlinear equation can be manipulated so as to give
\begin{equation}\label{eq:lameq}
w_i\,e^{w_i} =     \tau\,I_0\,e^{\tau\,y_i-q_i}\,,\quad \text{with}\;\; w_i = t_i+\tau\,y_i-q_i\,.
\end{equation}
Equations of the form in \eqref{eq:lameq} admit solutions that can be expressed in closed-form in terms of the so-called Lambert $W$ function \cite{lambert}. In particular, when the right-hand side is non-negative - which is our case as $\tau \, I_0 e^{\tau\,y_i-q_i} \in \R_{++}$ - then the equation admits a unique solution given by
\begin{equation}
w_i = W\left(\tau\,I_0\,e^{\tau\,y_i-q_i}\right) \, .  
\end{equation}
It follows that problem \eqref{eq:zupdct} admits the unique solution
\begin{equation}
t_i^{(k+1)} = -(\tau\,y_i-q_i) + W\left(\tau\,I_0\,e^{\tau\,y_i-q_i}\right) \, .
\end{equation}

\vspace{0.2cm}\emph{Update of} $\bm{t}_3$. It comes from \eqref{eq:subz2} that the  $\bm{t}_3$-update problem reads
\begin{equation}
\bm{t_3}^{(k+1)}\;{\in}\;\argmin_{\bm{t}_3\,\in\,\R_+^n}\,\|\bm{t}_3-\bm{q}_3^{(k)}\|_2^2 \,,
\end{equation}
that is $\bm{t_3}^{(k+1)}$ is given  by the unique Euclidean projection of vector $\bm{q}_3^{(k)}$ onto the non-negative orthant $\R_+^n$, which admits the following component-wise closed-form expression:
\begin{equation}\label{eq:getz4i}
t_{3,i}^{(k+1)}\;{=}\;\max\left\{q_{3,i}^{(k)},0\right\}\,,\quad i=1,\ldots,n\,.
\end{equation}

\section{Computed examples}
\label{sec:test}
In this section, we evaluate the performance of the proposed Poisson Whiteness Principle (\ref{eq:PWP}) for the automatic selection of the regularization parameter $\mu$ in the TV-KL model in \eqref{eq:TVKL} employed for the image restoration and CT image reconstruction tasks.

The proposed strategy is compared with the \ref{eq:ADP} and the \ref{eq:NEDP}. The considered parameter selection rules are applied \textit{a posteriori}. In other words, the TV-KL model is solved on a grid of different $\mu$-values; then, for each output image, we compute the discrepancy function, involved in the ADP and NEDP, and the whiteness measure, which is used for the PWP. 
The $\mu$-values selected by the ADP, NEDP and the PWP will be denoted by $\mu^{(A)}$, $\mu^{(NE)}$ and $\mu^{(W)}$, respectively.

The quality of the output image $\hat{\bm{x}}$ with respect to the original image $\bar{\bm{x}}$ is measured by means of two scalar measures, namely the Structural Similarity Index (SSIM) \cite{Wang(2204)} and the Signal-to-Noise-Ratio (SNR) defined by
\begin{equation}
\text{SNR}(\hat{\bm{x}},\bar{\bm{x}}) = 10\log_{10} \frac{||\bar{\bm{x}}- \mathrm{E}[\bar{\bm{x}}]||_{2}^{2}}{||\bar{\bm{x}}-\hat{\bm{x}}||_{2}^{2}}.
\end{equation}
In the performed tests, the ADMM iterations are stopped as soon as
\begin{equation}
\delta_{\bm{x}}^{(k)} = \frac{\|\bm{x}^{(k)}-\bm{x}^{(k-1)}\|_2}{\|\bm{x}^{(k-1)}\|_2}<10^{-6}\,,\qquad k\in\mathbb{N}\setminus\{0\}\,,
\end{equation}
while the ADMM penalty parameter $\beta$ is set manually so as to fasten the convergence of the alternating scheme.

\subsection{Image restoration}

We start testing our proposal on the image restoration task, and consider two test images, namely \texttt{satellite} ($256\times 256$) and \texttt{cells} ($236\times 236$), with pixel values between 0 and 1, shown in Figures \ref{fig:true}a, \ref{fig:true}b.

\begin{figure}
	\centering
	\setlength{\tabcolsep}{0.3pt}
	\begin{tabular}{cccc}
		\includegraphics[height=3.1cm]{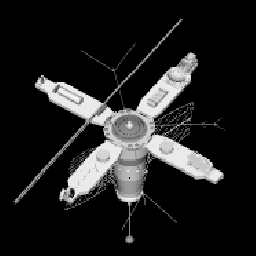}   &\includegraphics[height=3.1cm]{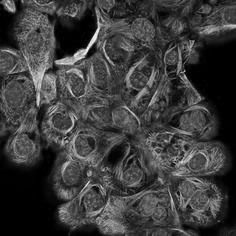}  &\includegraphics[height=3.1cm]{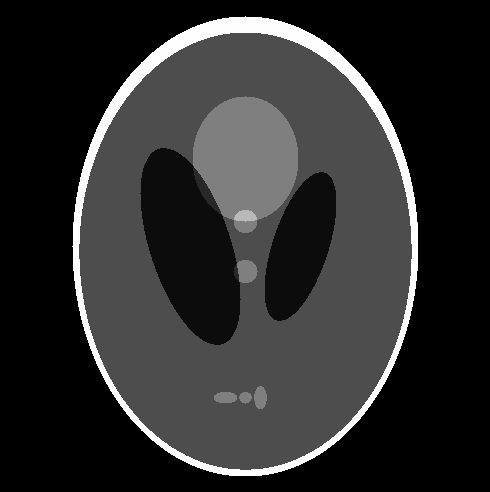} &\includegraphics[height=3.1cm]{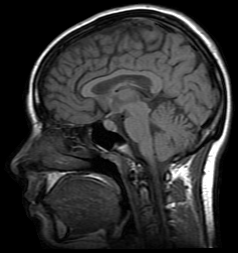}\\
		(a)&(b)&(c)&(d)
	\end{tabular}
	\caption{From left to right: original \texttt{satellite} ($256\times 256$), \texttt{cells} ($236\times 236$), \texttt{shepp logan ($500\times 500$)} and \texttt{brain} ($238\times 253$) test images considered for the numerical experiments.}
	\label{fig:true}
\end{figure}

We simulate the acquisition process by multiplying the original images by a factor $\kappa\in\mathbb{R}_{++}$ representing the maximum number of photons hitting the image domain, in expectation. Clearly, the lower the value of $\kappa$ the noisier the data, yielding a more difficult image restoration problem. Then, the resulting images are been corrupted by space-invariant Gaussian blur, with blur kernel generated by the Matlab routine \texttt{fspecial}, which is characterized by two parameters: the \texttt{band} parameter, representing the side length (in pixels) of the square support of the kernel, and \texttt{sigma}, that is the standard deviation (in pixels) of the isotropic bivariate Gaussian distribution defining the kernel in the continuous setting. In our tests, we set \texttt{band}=5, \texttt{sigma}=1. Then, we add a constant emission background $\bm{b}$ equal to $2\times10^{-3}$, obtaining what we define as $\bm{\bar{\lambda}} = \bm{\mathrm{H}}\bm{\bar{x}} + \bm{b}$. Finally, the observed image $\bm{y} = \bm{\mathrm{poiss}}(\bm{\bar{\lambda}})$ is pseudo-randomly generated by a \textit{m}-variate independent Poisson realization with mean vector $\bm{\bar{\lambda}}$.

In Figure \ref{fig:sat_curve}c,d we show the Whiteness function $W(\mu)$, as defined in \ref{eq:PWP}, for the first image \texttt{satellite} and $\kappa=5$ (left) and $\kappa= 10$ (right). 
The vertical dashed red lines correspond to the minimum of the function $W(\mu)$, i.e. to the chosen values of $\mu$ according to the Poisson Whiteness Principle, namely $\mu^{(W)}$.

The black curves in Figure \ref{fig:sat_curve}a,b represent the discrepancy function $\mathcal{D}(\mu,\bm{y})$ as defined in \eqref{eq:discrep}, while the green and magenta dashed lines represent the discrepancy values $\Delta^{(A)}$ and $\Delta^{(NE)}(\mu)$ as defined in \ref{eq:ADP} and \ref{eq:NEDP}, respectively. 

In Figure \ref{fig:sat_curve}e,f, we show the SNR (in blue) and SSIM (in orange) values achieved for different $\mu$ values with $\kappa=5,10$. The red, green and magenta vertical lines correspond to the $\mu$ values chosen with the newly proposed method and the two considered versions of the DP. We remark that the $\mu$ values selected by the discrepancy principles correspond to the intersection of $\mathcal{D}(\mu,\bm{y})$ and $\Delta^{(A)}$,$\Delta^{(NE)}(\mu)$, respectively. Note that, in the low-count regime, the PWP achieves higher values of SNR and SSIM if compared to the ADP and NEDP.

Furthermore, at the bottom of Figure \ref{fig:sat_curve}, we report, for different counting regimes $\kappa$, the values of the selected $\mu$ , the SNR and SSIM values for the three considered strategies. 
For each $\kappa$, the highest values of SNR and SSIM are reported in bold. As already observed in Figure \ref{fig:sat_curve}e,f, the PWP outperforms the ADP and NEDP in terms of SNR and SSIM for the low-middle counts acquisitions (up to $\kappa=50$). For the higher counts NEDP and PWP achieve similar quality measures, with NEDP being slightly better.

For a visual comparison, in Figure \ref{fig:satellite}, we show the observed images and the output restorations obtained by employing  ADP, NEDP and PWP for $\kappa=5$ (top row) and $\kappa=10$ (bottom row). In both cases, the NEDP and the PWP return similar results, with the latter being more capable of preserving the original contrast in the image. On the other hand, the output images obtained by selecting $\mu$ according to ADP are strongly over-regularized.


\begin{figure}
	\centering
	\setlength{\tabcolsep}{2pt}
	\begin{tabular}{ccc}
		\includegraphics[height=3.9cm]{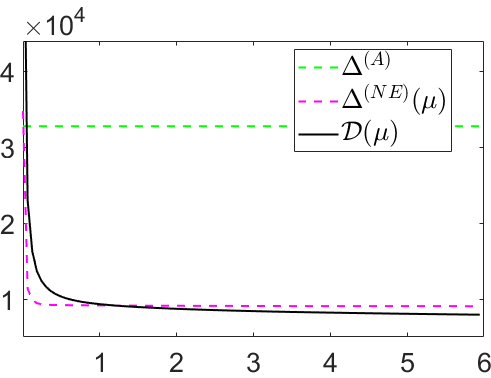}   &\includegraphics[height=3.9cm]{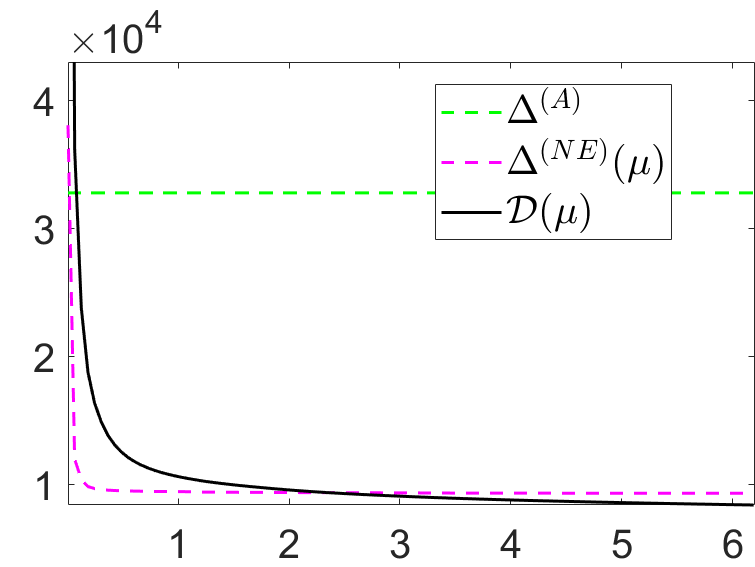} \\
		(a)&(b)\\
		\includegraphics[height=3.9cm]{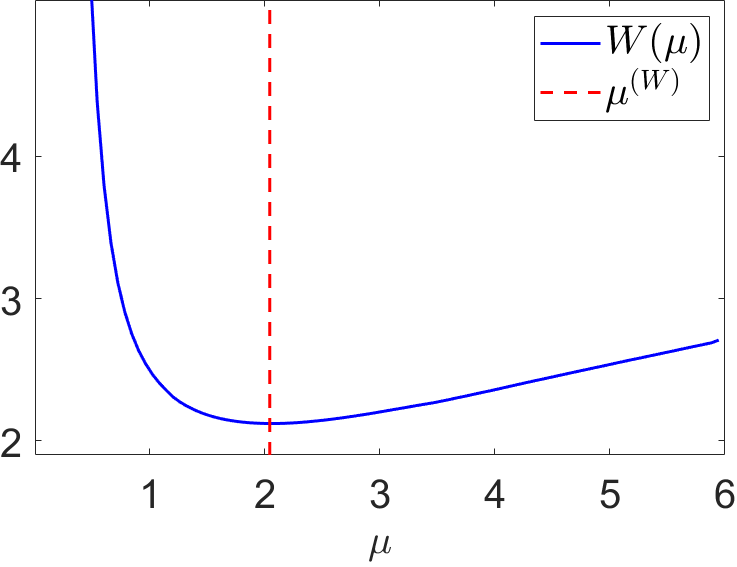} 
		&\includegraphics[height=3.9cm]{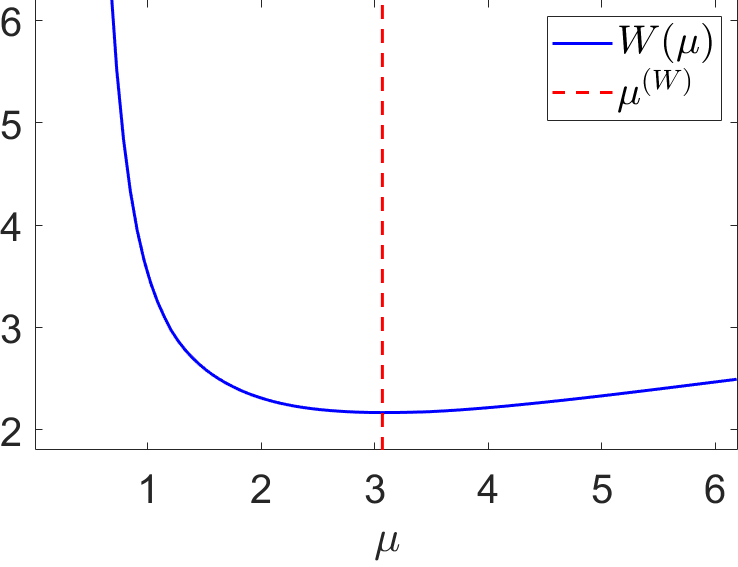} \\
		(c)&(d)\\
		\includegraphics[height=3.9cm]{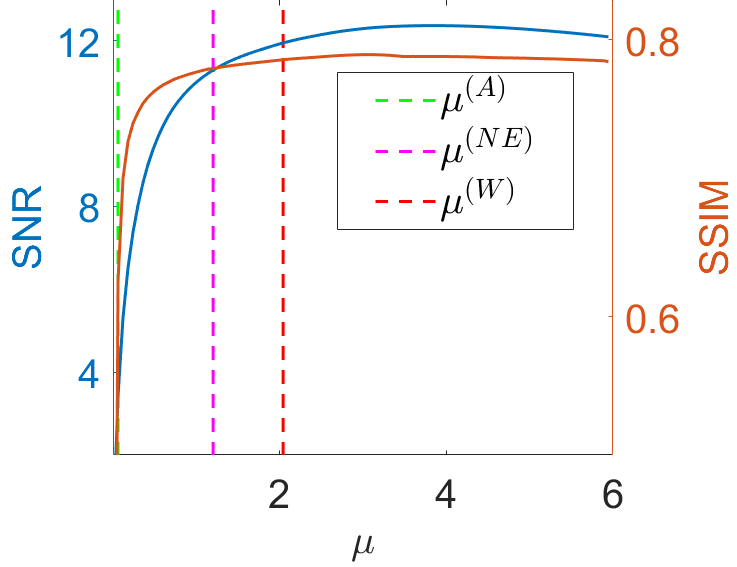}
		&
		\includegraphics[height=3.9cm]{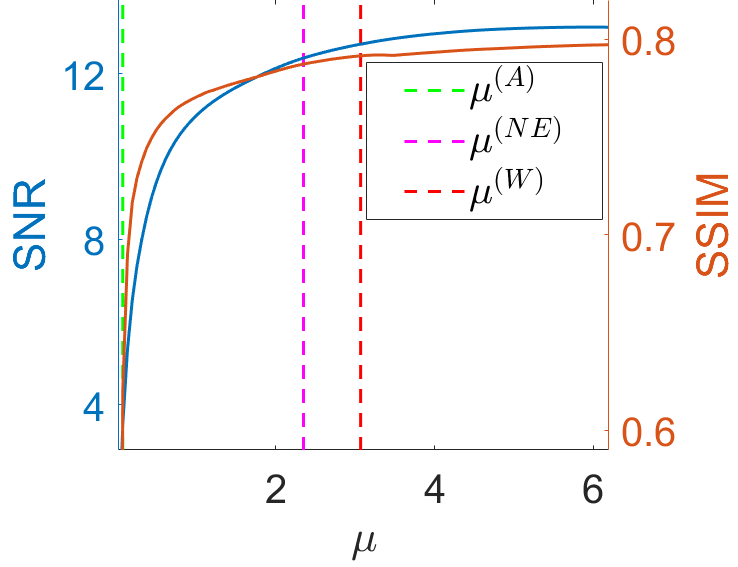}     \\
		(e)&(f)
	\end{tabular}

	\setlength{\tabcolsep}{4pt}
	\begin{tabular}{r||rrr||rrr||rrr}
		&
		\multicolumn{3}{c}{ADP}&\multicolumn{3}{c}{NEDP}&\multicolumn{3}{c}{PWP}\\
		\hline\hline
		$\kappa$& ${\mu}^{(A)}$&SNR&SSIM& ${\mu}^{(NE)}$&SNR&SSIM& ${\mu}^{(W)}$&SNR&SSIM\\
		\hline\hline
		{1.5}&$2\times 10^{-5}$&-0.001&0.009&0.841&10.270&0.786&1.201&\textbf{10.618}&\textbf{0.787}\\
		{5}&0.065&3.408&0.625 &1.205&11.286&0.779&2.045&\textbf{11.944}&\textbf{0.785}\\
		{10}&0.068&3.508&0.6199&2.348&12.384&0.787&3.068&\textbf{12.719}&\textbf{0.791}\\
		{20}&0.188&6.580&0.708&3.848&13.179&0.792&4.388&\textbf{13.328}&\textbf{0.79}4\\
		{50}&0.380&8.688&0.724 &6.800&14.206&0.805&7.460&\textbf{14.313}&\textbf{0.808}\\
		{100}&0.760&10.574&0.742&11.380&\textbf{15.017}&\textbf{0.823}&11.080&14.983&0.822\\
		{1000}&8.260&14.747&0.805 &60.760&\textbf{17.540}&\textbf{0.862}&45.220&17.225&0.857\\
	\end{tabular}
	\caption{Test image \texttt{satellite}. From top to bottom: discrepancy curves, whiteness curves and achieved SNR/SSIM for $\kappa=5$ (left) and $\kappa=10$ (right). Output $\mu$- and SNR/SSIM values
		obtained by the ADP, the NEDP and the PWP for
		different $\kappa$.
		%
	}
	\label{fig:sat_curve}
\end{figure}

\begin{figure}
	\centering
	\renewcommand{\arraystretch}{0.3}
	\setlength{\tabcolsep}{8pt}
	\begin{tabular}{cccc}
		$\bm{y}$ &ADP& NEDP&PWP\\  
		\begin{overpic}[height=2.6cm]{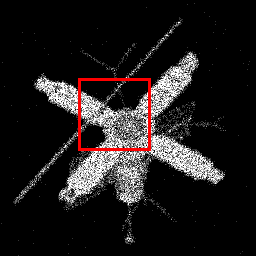}
			\put(60,1){\color{red}%
				\frame{\includegraphics[scale=0.15]{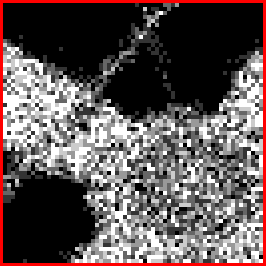}}}
		\end{overpic} & \begin{overpic}[height=2.6cm]{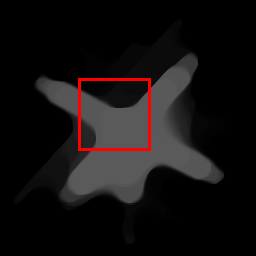}
			\put(60,1){\color{red}%
				\frame{\includegraphics[scale=0.15]{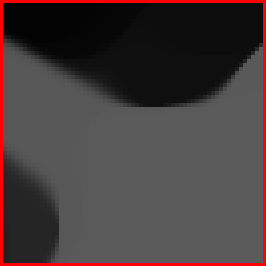}}}
		\end{overpic}   &  \begin{overpic}[height=2.6cm]{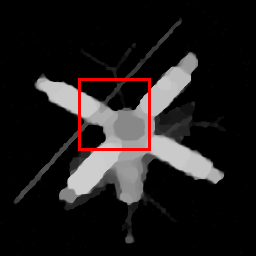}
			\put(60,1){\color{red}%
				\frame{\includegraphics[scale=0.15]{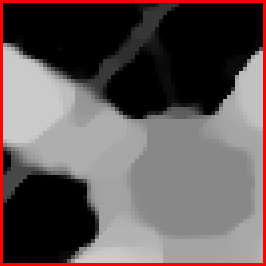}}}
		\end{overpic} &  \begin{overpic}[height=2.6cm]{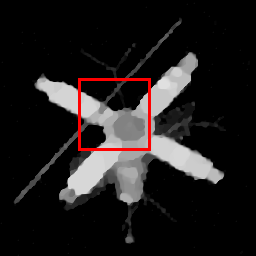}
			\put(60,1){\color{red}%
				\frame{\includegraphics[scale=0.15]{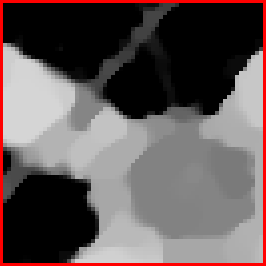}}}
		\end{overpic} \\
		\begin{overpic}[height=2.6cm]{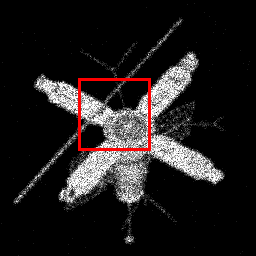}
			\put(60,1){\color{red}%
				\frame{\includegraphics[scale=0.15]{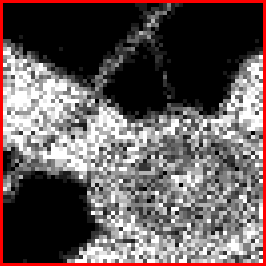}}}
		\end{overpic} & \begin{overpic}[height=2.6cm]{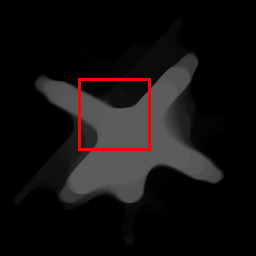}
			\put(60,1){\color{red}%
				\frame{\includegraphics[scale=0.15]{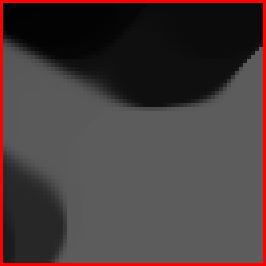}}}
		\end{overpic}   &  \begin{overpic}[height=2.6cm]{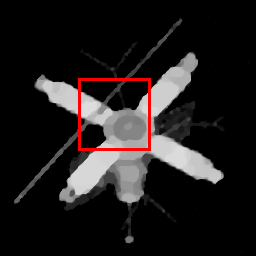}
			\put(60,1){\color{red}%
				\frame{\includegraphics[scale=0.15]{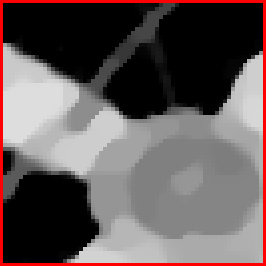}}}
		\end{overpic} &  \begin{overpic}[height=2.6cm]{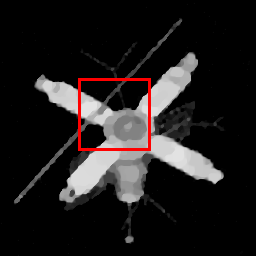}
			\put(60,1){\color{red}%
				\frame{\includegraphics[scale=0.15]{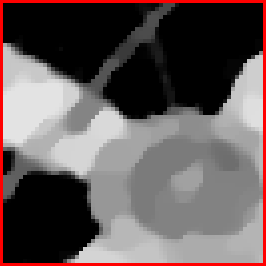}}}
		\end{overpic} 
	\end{tabular}
	\caption{Test image \texttt{satellite}. From left to right: observed image $\bm{y}$, reconstruction using the ADP, NEDP and PWP for $\kappa=5$ (top row) and $\kappa=10$ (bottom row).}
	\label{fig:satellite}
\end{figure}

For the second test image, \texttt{cells}, we report in Figure \ref{fig:cells_curve} the behavior of the discrepancy function $\mathcal{D}(\mu,\b,{y})$, of the Whiteness function $W(\mu)$ and of the SNR/SSIM curves obtained by applying the NEDP, the ADP and the PWP, for $\kappa=5$ (left) and $\kappa=10$ (right). The PWP returns larger quality measures, as it 
is the closest to the maximum SNR/SSIM achievable.

From the table reported at the bottom of Figure \ref{fig:cells_curve}, we note that the proposed $\mu$-selection criterion returns restored images outperforming the ones obtained via the NEDP and ADP both in terms of SNR and SSIM, for every $\kappa\geq 5$. For $\kappa=1.5$ the SNR and SSIM values of the PWP restoration are slightly lower, but very similar, to the one obtained with NEDP, while in all the other cases the difference between the PWP and the NEDP the difference is more marked. 

The restored images in Figure \ref{fig:cells} reflect the values recorded in the tables: the output of the PWP preserve more details and the original contrast if compared to NEDP, while the ADP restoration seems to be less subject to over-regularization if compared to the results obtained on the test image \texttt{satellite}. This can be ascribed to the number of zeros in the image, being significantly smaller in \texttt{cells}. 

\begin{figure}
	\centering
	\setlength{\tabcolsep}{2pt}
	\begin{tabular}{ccc}
		\includegraphics[height=3.9cm]{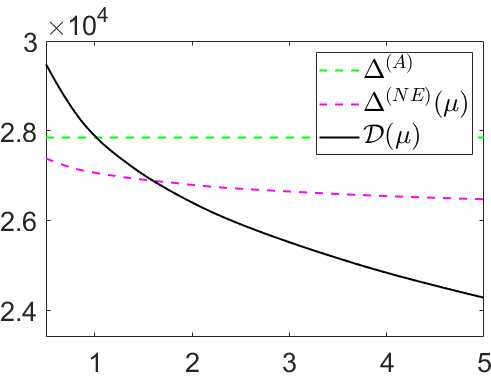}  &\includegraphics[height=3.9cm]{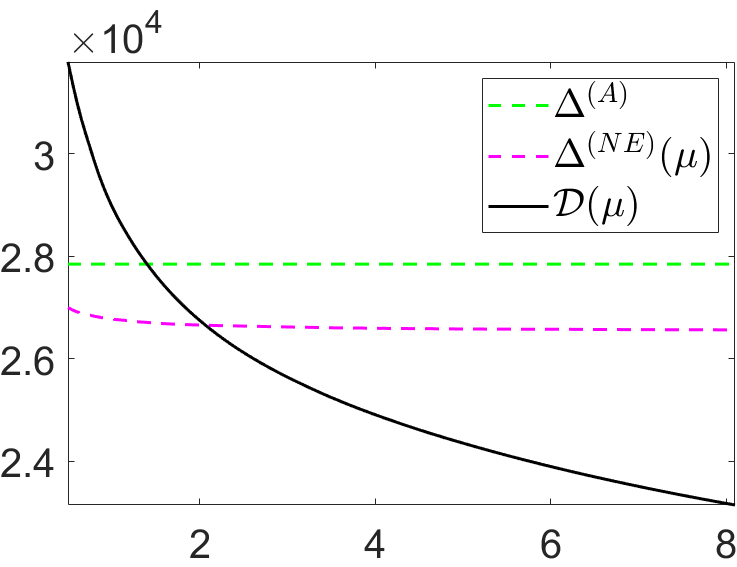}  \\
		(a)&(b)\\
		\includegraphics[height=3.9cm]{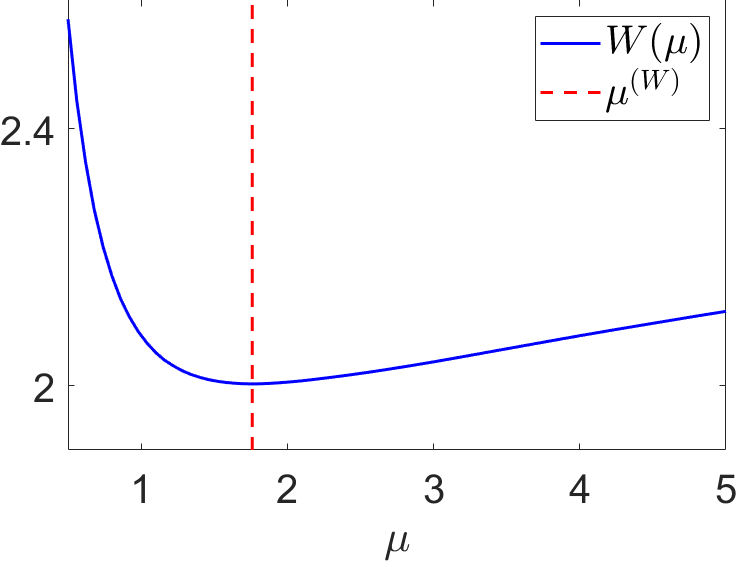}   &\includegraphics[height=3.9cm]{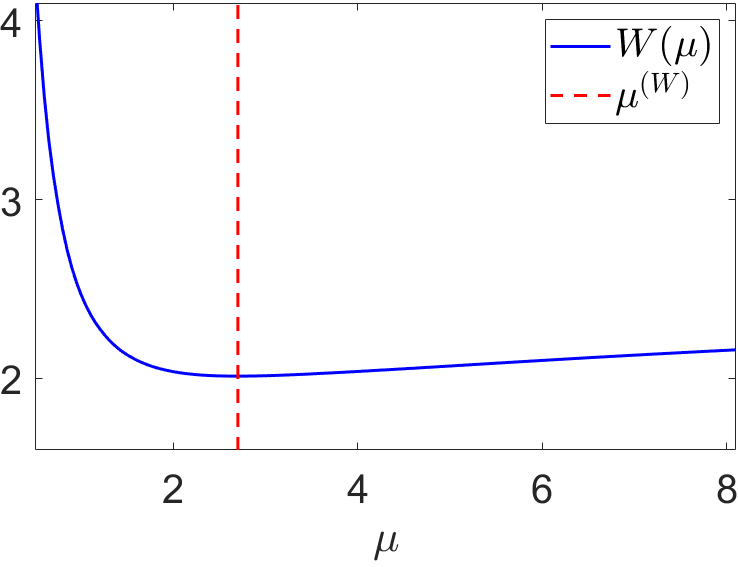}  \\
		(c)&(d)\\
		\includegraphics[height=3.9cm]{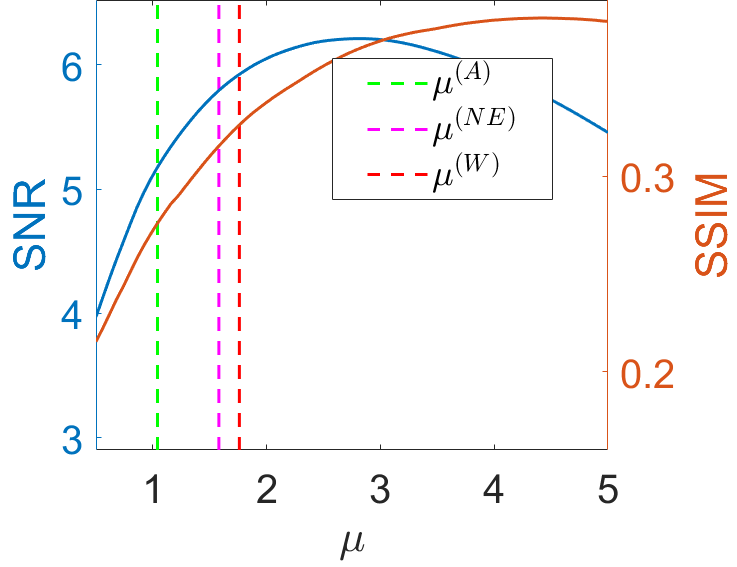}&
		\includegraphics[height=3.9cm]{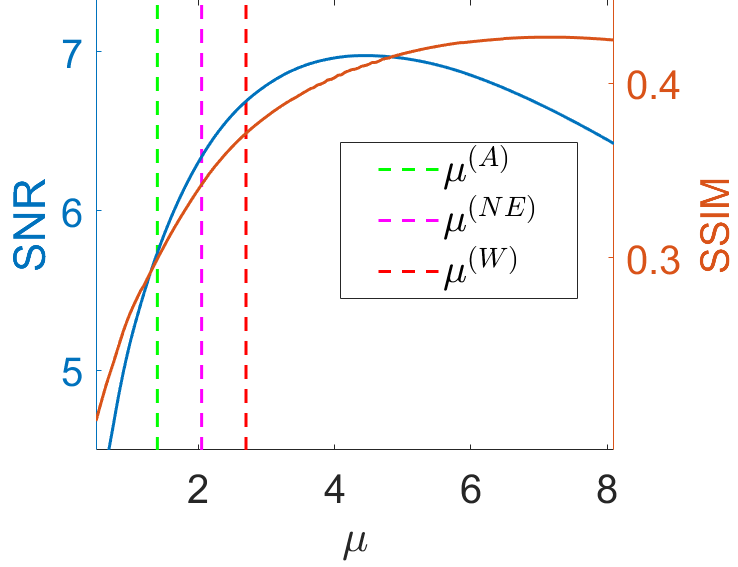}\\
		(e)&(f)\\
	\end{tabular}
	\setlength{\tabcolsep}{4pt}
	\begin{tabular}{r||rrr||rrr||rrr}
		&
		\multicolumn{3}{c}{ADP}&\multicolumn{3}{c}{NEDP}&\multicolumn{3}{c}{PWP}\\
		\hline\hline
		$\kappa$& ${\mu}^{(A)}$&SNR&SSIM& ${\mu}^{(NE)}$&SNR&SSIM& ${\mu}^{(W)}$&SNR&SSIM\\
		\hline\hline
		1.5&$7\times 10^{-5}$&0.004&0.077 &0.875&\textbf{4.625}&\textbf{0.313}&0.850&4.601&0.311\\
		5&1.040&5.176&0.276&1.580&5.794&0.315&1.760&\textbf{5.924}&\textbf{0.326}\\
		10&1.400&5.737&0.299 &2.060&6.347&0.342&2.720&\textbf{6.695}&\textbf{0.372}\\
		20&2.280&6.626&0.363&3.600&7.316&0.417&4.200&\textbf{7.503}&\textbf{0.433}\\
		50&4.500&7.830&0.452&6.600&8.337&0.493&7.440&\textbf{8.470}&\textbf{0.505}\\
		100&7.920&8.735&0.532 &10.680&9.071&0.560&12.000&\textbf{9.186}&\textbf{0.571}\\
		1000&45.000&11.075&0.717&52.140&11.207&0.730&54.660&\textbf{11.248}&\textbf{0.733}\\
	\end{tabular}
	\caption{Test image \texttt{cells}. From top to bottom: discrepancy curves, whiteness curves and achieved SNR/SSIM for $\kappa=5$ (left) and $\kappa=10$ (right). Output $\mu$- and SNR/SSIM values
		obtained by the ADP, the NEDP and the PWP for
		different $\kappa$.
		%
	}
	\label{fig:cells_curve}
\end{figure}

\begin{figure}
	\centering
	\renewcommand{\arraystretch}{1}
	\setlength{\tabcolsep}{2pt}
	\begin{tabular}{cccc}
		$\bm{y}$ &ADP& NEDP&PWP\\
		\begin{overpic}[height=2.6cm]{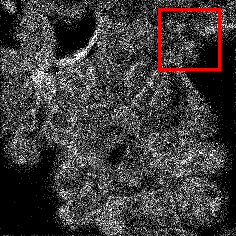}
			\put(1,1){\color{red}%
				\frame{\includegraphics[scale=0.2]{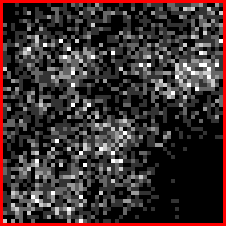}}}
		\end{overpic}    &
		\begin{overpic}[height=2.6cm]{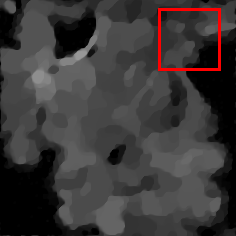}
			\put(1,1){\color{red}%
				\frame{\includegraphics[scale=0.2]{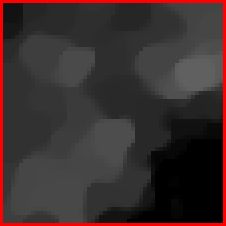}}}
		\end{overpic}   &  \begin{overpic}[height=2.6cm]{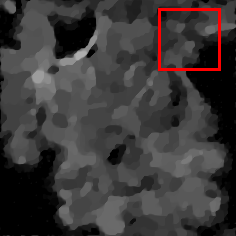}
			\put(1,1){\color{red}%
				\frame{\includegraphics[scale=0.2]{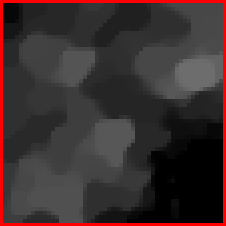}}}
		\end{overpic} &  \begin{overpic}[height=2.6cm]{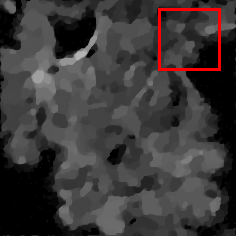}
			\put(1,1){\color{red}%
				\frame{\includegraphics[scale=0.2]{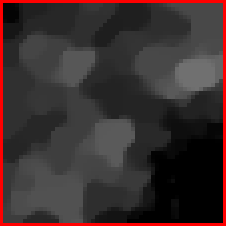}}}
		\end{overpic} \\
		\begin{overpic}[height=2.6cm]{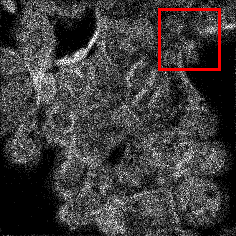}
			\put(1,1){\color{red}%
				\frame{\includegraphics[scale=0.2]{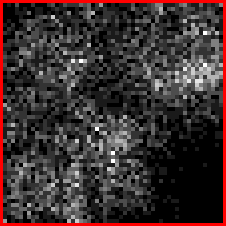}}}
		\end{overpic}  
		&
		\begin{overpic}[height=2.6cm]{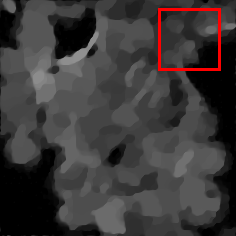}
			\put(1,1){\color{red}%
				\frame{\includegraphics[scale=0.2]{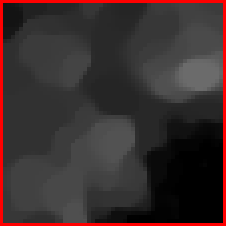}}}
		\end{overpic}   &  \begin{overpic}[height=2.6cm]{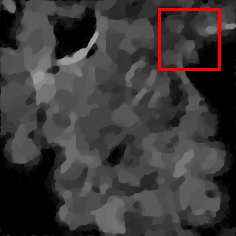}
			\put(1,1){\color{red}%
				\frame{\includegraphics[scale=0.2]{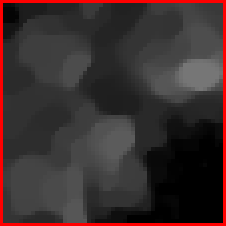}}}
		\end{overpic} &  \begin{overpic}[height=2.6cm]{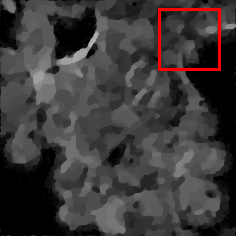}
			\put(1,1){\color{red}%
				\frame{\includegraphics[scale=0.2]{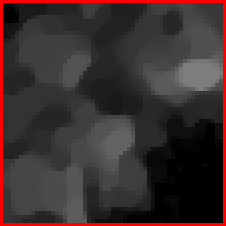}}}
		\end{overpic} \\
	\end{tabular}
	\caption{Test image \texttt{cells}. From left to right: observed image $\bm{y}$, reconstruction using the ADP, NEDP and PWP for $\kappa=5$ (top row) and $\kappa=10$ (bottom row).}
	\label{fig:cells}
\end{figure}

\subsection{CT image reconstruction}

For the CT reconstruction problem we consider the test images \texttt{shepp logan} ($500\times 500$, pixel size = $0.2$mm) and \texttt{brain} ($238\times 253$, pixel size=$0.4$mm), with pixel values between 0 and 1, shown Figures \ref{fig:true}c, \ref{fig:true}d, respectively. The acquisition process of the fan beam CT setup, i.e. the projection operator $\bm{\mathrm{H}}$, is built using the ASTRA Toolbox \cite{ASTRA} with the following parameters: 180 equally spaced angles of projections (from $0$ to $2\pi$), a detector with $500$ pixels (detector pixel size = $1/3$mm), distance between the source and the center of rotation = $300$mm, distance between the center of rotation and the detector array = $200$mm. Then, according to \eqref{eq:IRCTIR}, we take the exponential of $-\bm{\mathrm{H}\bar{x}}$ and multiply it by a factor $I_{0}\in\mathbb{N}\setminus \{0\}$ that plays the role of $\kappa$ in the restoration scenario and represents the maximum emitted photon counts, i.e., the maximum number of photons that can reach each detector pixel if the X-rays are not attenuated. In the CT tests, we consider the background emission $\bm{b}= \bm{0}$ so that the solution of \eqref{eq:lameq} can be expressed in closed-form in terms of the Lambert function. We thus compute the noise-free data $\bm{\bar{\lambda}} = I_{0}e^{-\bm{\mathrm{H}}\bm{\bar{x}}}$, while the acqisition $\bm{y} = \bm{\mathrm{poiss}}(\bm{\bar{\lambda}})$ is obtained by generating an $m$-variate independent Poisson realization with mean vector $\bm{\bar{\lambda}}$.

In analogy to the restoration case, in Figure \ref{fig:phantom_curve}, we report for the test image \texttt{shepp logan} the curve of the discrepancy function $\mathcal{D}(\mu,\bm{y})$, as well as the Whiteness curve $W(\mu)$ and the curves of the SNR and SSIM for the limiting values $I_0$, i.e. $I_{0}=1.5$ (left) and $I_{0}=1000$ (right). In the case of $I_{0}=1.5$ the SNR/SSIM values achieved by ADP and NEDP are significanlty far from te optimal ones. On the other hand, PWP one is very close to the maximum of both the SNR and the SSIM. For $I_{0}=1000$, the NEDP and the ADP select the same $\mu$, which allows to achieve a larger SSIM with respect to the one obatined by PWP, while our method still outperforms the other in terms of SNR. 
From the table at the bottom of Figure \ref{fig:phantom_curve}, we observe that the PWP outperforms the ADP and the NEDP in terms of SNR for each $I_{0}$ value, while the NEDP returns slightly better results in terms of SSIM for high-count acquisitions. 

The reconstruction results shown in Figure \ref{fig:phantom} reflect the behavior of the plots. More specifically, for $I_{0} = 1.5$ the ADP reconstruction appears to be over-regularized; NEDP allows to reconstruct only the central ellipsis, which appear to be merged; finally, in the PWP reconstruction the two ellipsis are more visible and the white edge of the phantom is sharper. In the case of $I_{0} = 1000$, the three reconstructions are similar, with the PWP being more capable of separating the three fine details highlighted in the super-imposed close-up. 
\begin{figure}
	\centering
	\setlength{\tabcolsep}{2pt}
	\begin{tabular}{ccc}
		\includegraphics[height=3.9cm]{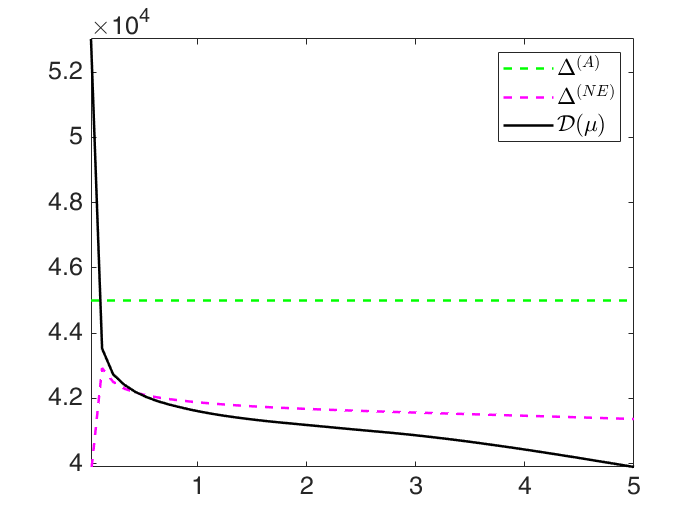}   &\includegraphics[height=3.9cm]{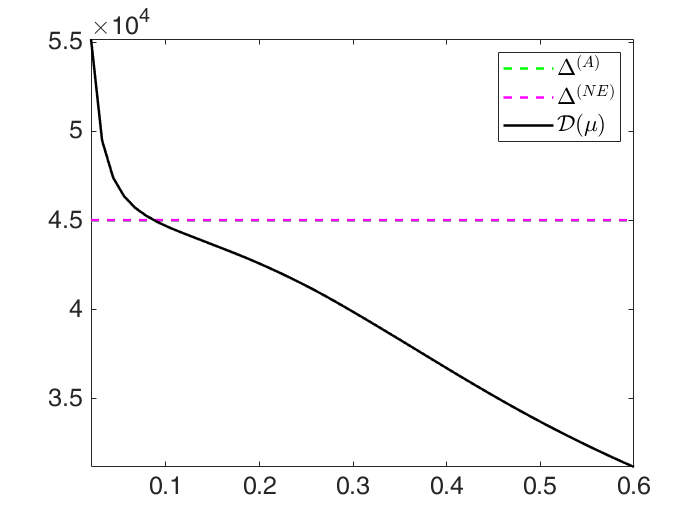} \\
		(a)&(b)\\
		\includegraphics[height=3.9cm]{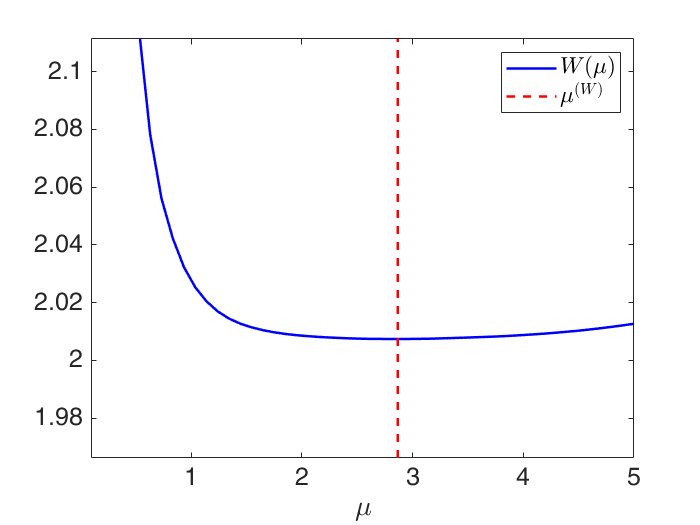}   &\includegraphics[height=3.9cm]{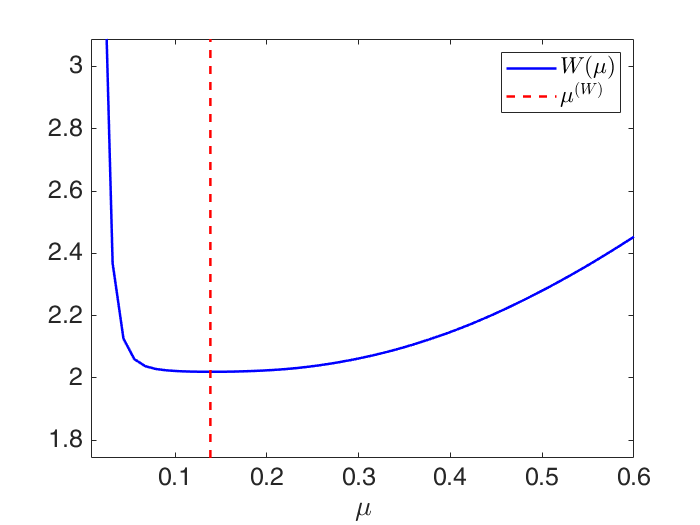} \\
		(c)&(d)\\
		\includegraphics[height=3.9cm]{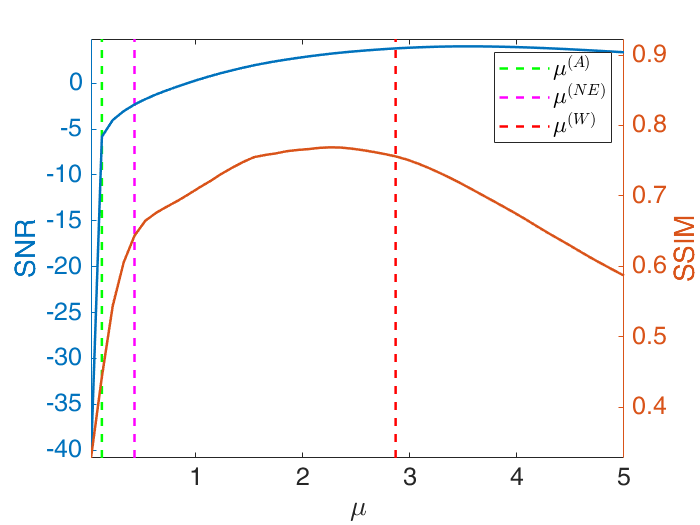}&
		\includegraphics[height=3.9cm]{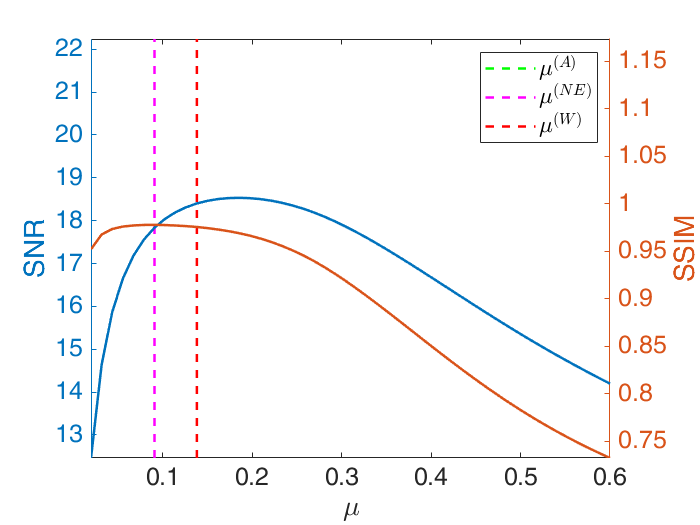} \\
		(e) & (f)
	\end{tabular}
	\setlength{\tabcolsep}{4pt}
	\begin{tabular}{r||rrr||rrr||rrr}
		&
		\multicolumn{3}{c}{ADP}&\multicolumn{3}{c}{NEDP}&\multicolumn{3}{c}{PWP}\\
		\hline\hline
		$I_0$& ${\mu}^{(A)}$&SNR&SSIM& ${\mu}^{(NE)}$&SNR&SSIM& ${\mu}^{(W)}$&SNR&SSIM\\
		\hline\hline
		1.5&0.122&-5.838&0.443 &0.426&-2.329&0.643&2.865&\textbf{3.785}&\textbf{0.756}\\
		5&5.555&2.899&0.455&0.684&3.569&0.793&1.363&\textbf{5.997}&\textbf{0.816}\\
		10&3.351&4.449&0.515 &0.733&6.403&0.853&1.024&\textbf{7.441}&\textbf{0.856}\\
		20&1.522&8.976&0.755 &0.530&8.550&\textbf{0.889}&0.861&\textbf{9.786}&0.883\\
		50&0.564&\textbf{11.626}&\textbf{0.974} &0.352&10.717&0.992&0.564&\textbf{11.626}&\textbf{0.974}\\
		100&0.322&13.141&0.944&0.261&12.698&\textbf{0.945}&0.442&\textbf{13.436}&0.935\\
		1000&0.091&17.837&\textbf{0.977} & 0.091&17.837&\textbf{0.977}&0.138&\textbf{18.401}&0.975\\
	\end{tabular}
	\caption{Test image \texttt{shepp logan}. From top to bottom: discrepancy curves, whiteness curves and achieved SNR/SSIM for $I_{0}=1.5$ (left) and $I_{0}=1000$ (right). Output $\mu$- and SNR/SSIM values
		obtained by the ADP, the NEDP and the PWP for
		different $I_0$.
		%
	}
	\label{fig:phantom_curve}
\end{figure}

\begin{figure}
	\centering
	\renewcommand{\arraystretch}{0.3}
	\setlength{\tabcolsep}{7pt}
	\begin{tabular}{cccc}
		$\bm{y}$&ADP& NEDP&PWP\\  
		\includegraphics[scale = 0.15]{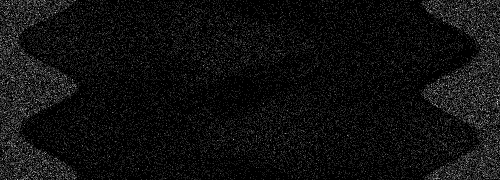}  & \begin{overpic}[height=2.4cm]{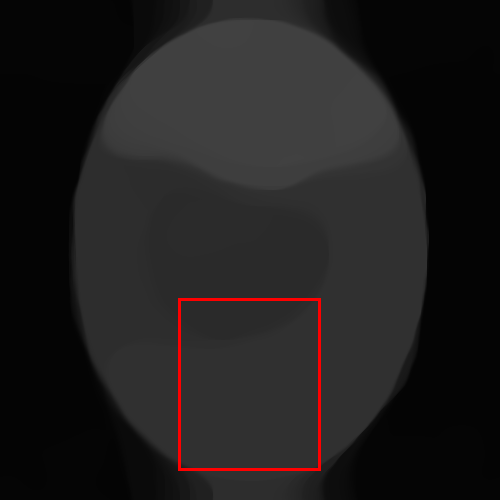}
			\put(60,1){\color{red}%
				\frame{\includegraphics[scale=0.07]{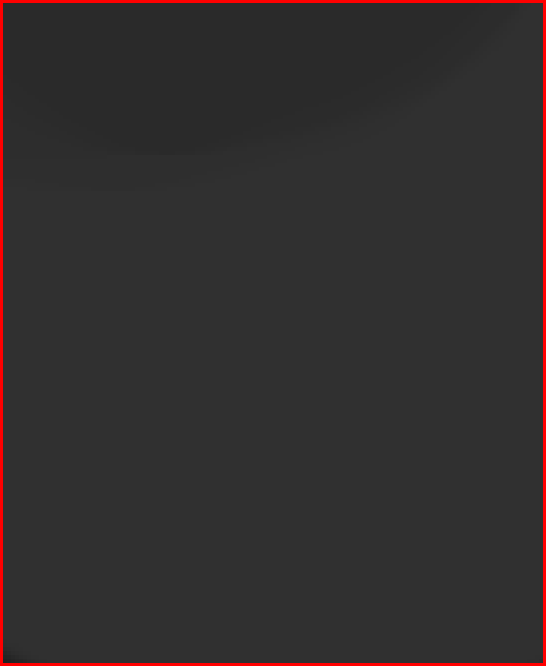}}}
		\end{overpic}   &  \begin{overpic}[height=2.4cm]{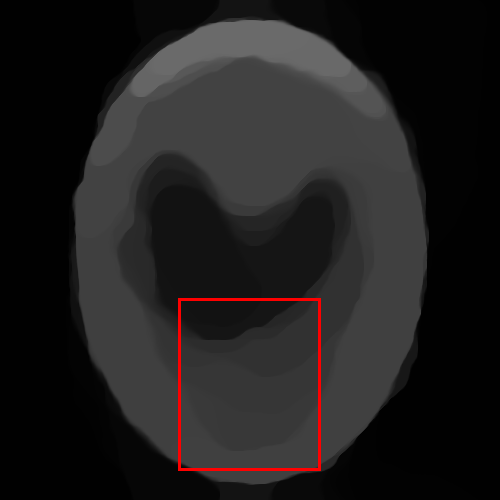}
			\put(60,1){\color{red}%
				\frame{\includegraphics[scale=0.07]{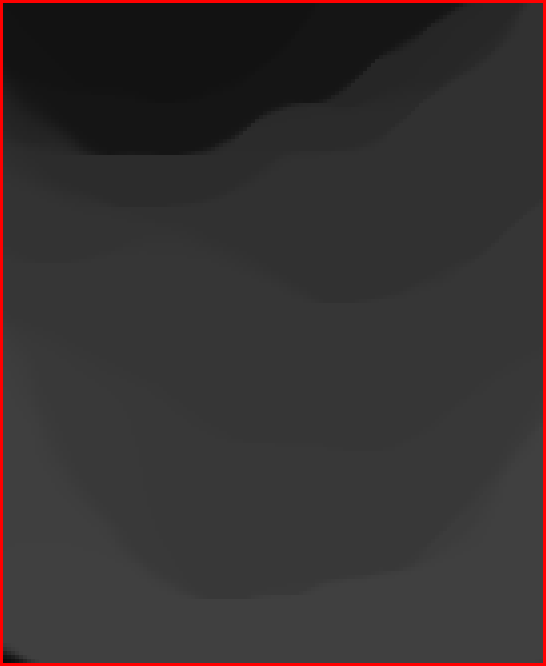}}}
		\end{overpic} &  \begin{overpic}[height=2.4cm]{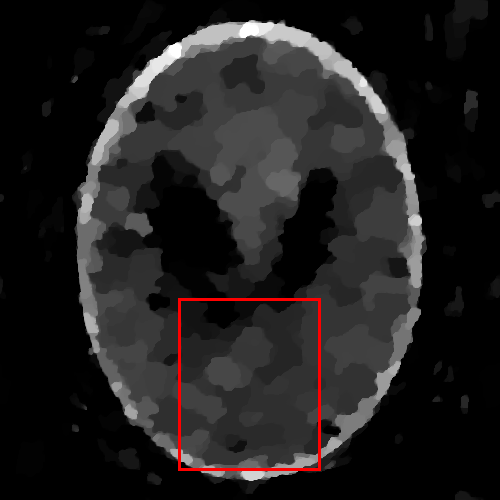}
			\put(60,1){\color{red}%
				\frame{\includegraphics[scale=0.07]{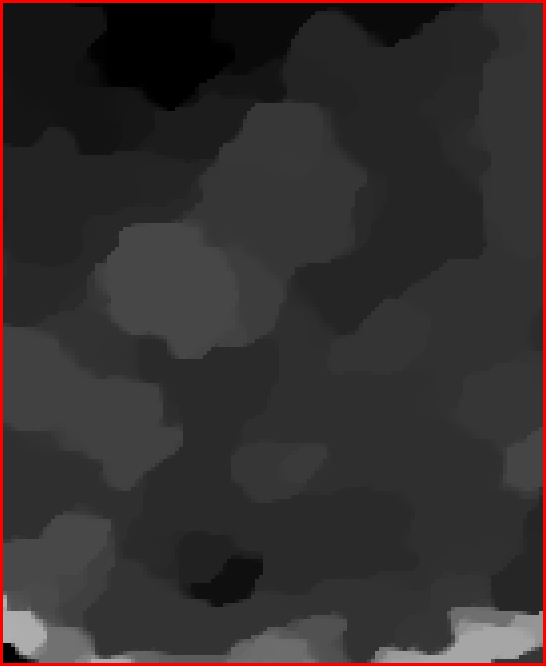}}}
		\end{overpic} \\
		
		\includegraphics[scale = 0.15]{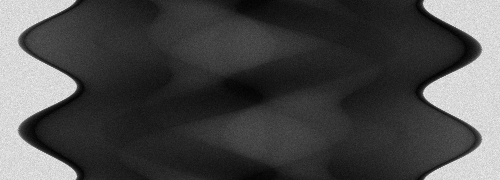} & \begin{overpic}[height=2.4cm]{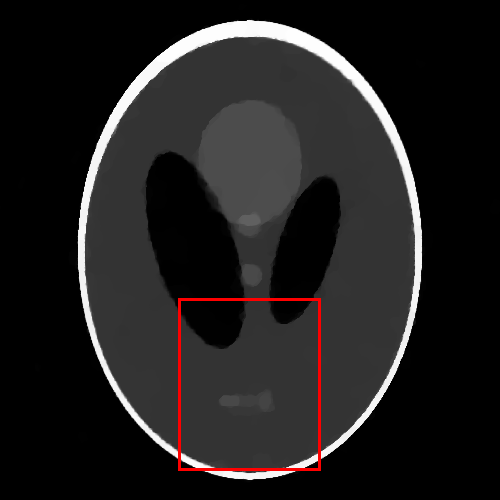}
			\put(60,1){\color{red}%
				\frame{\includegraphics[scale=0.07]{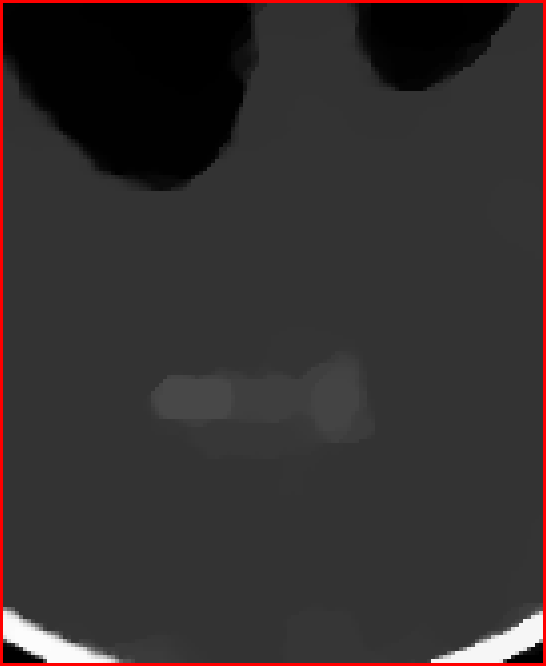}}}
		\end{overpic}   &  \begin{overpic}[height=2.4cm]{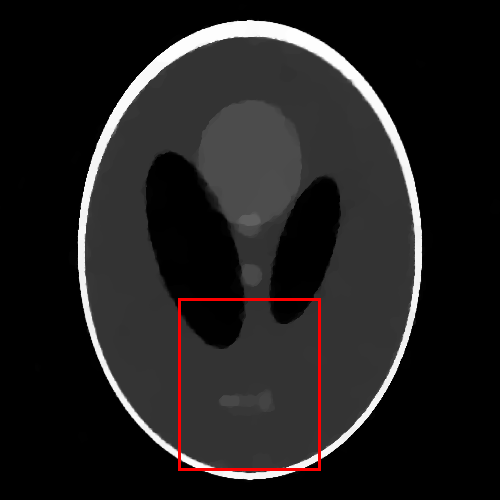}
			\put(60,1){\color{red}%
				\frame{\includegraphics[scale=0.07]{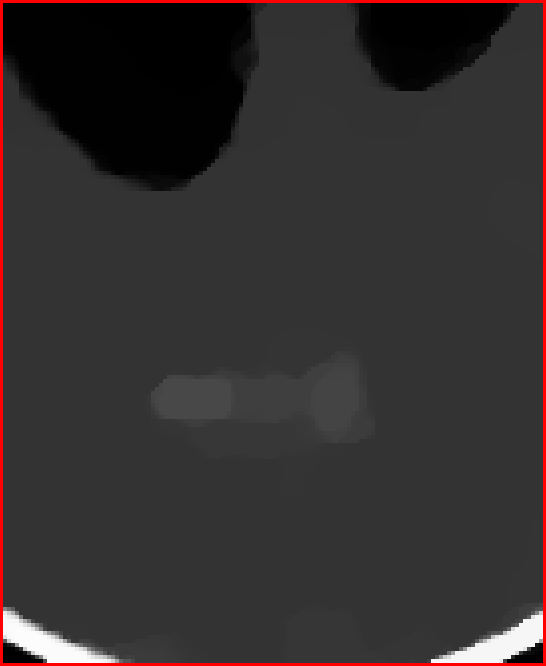}}}
		\end{overpic} &  \begin{overpic}[height=2.4cm]{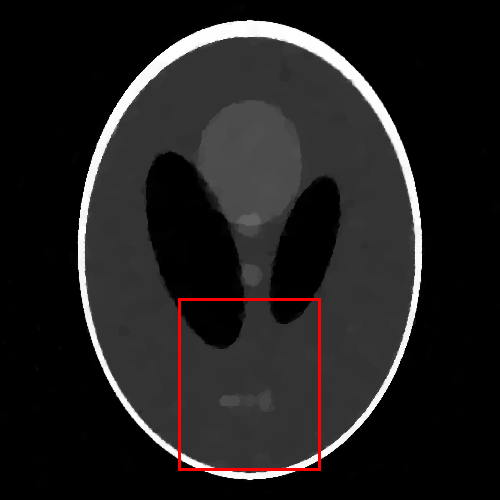}
			\put(60,1){\color{red}%
				\frame{\includegraphics[scale=0.07]{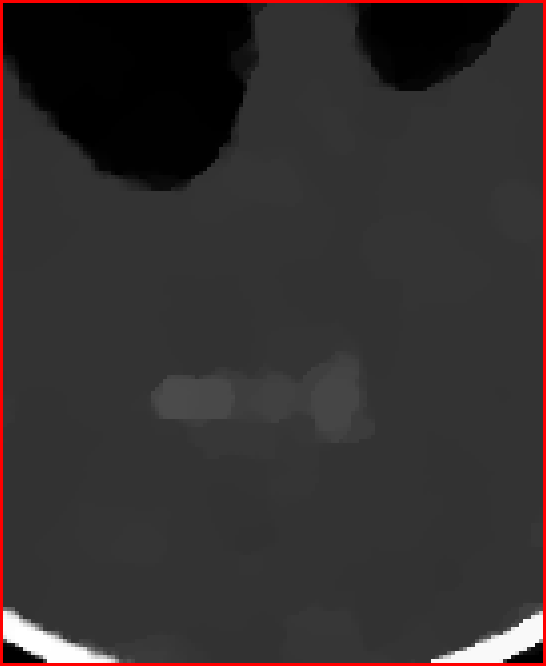}}}
		\end{overpic} 
		
	\end{tabular}
	\caption{Test image \texttt{shepp logan}. From left to right: observed data $\bm{y}$, reconstruction using the ADP, NEDP and PWP for $I_{0}=1.5$ (top row) and $I_{0}=1000$ (bottom row).}
	\label{fig:phantom}
\end{figure}

For the last test image, \texttt{brain}, we show in Figure \ref{fig:mri_curve} the behaviour of the discrepancy function $\mathcal{D}(\mu,\bm{y})$, of the Whiteness function $W(\mu)$, as well as of the SNR and SSIM values for $I_{0} = 1.5$ and $I_{0}=1000$. Note that the PWP achieves higher SNR and SSIM values compared to the ADP and NEDP for lower values of $I_{0}$. However, we observe that when considering higher values of $I_{0}$, the ADP reconstruction can outperform PWP for some of the considered doses. 

The reconstruction computed by ADP, NEDP and PWP are shown in Figure \ref{fig:mri}: we can see a higher level of details in the PWP reconstruction, both in the the low-dose and high-dose case. For $I_{0}=1.5$, only PWP is able to recover the upper part of the skull bone, while for $I_{0}=1000$ the difference mainly concerns the level of details present in the reconstruction, as shown in the close-ups.

\begin{figure}
	\centering
	\setlength{\tabcolsep}{2pt}
	\begin{tabular}{ccc}
		\includegraphics[height=3.9cm]{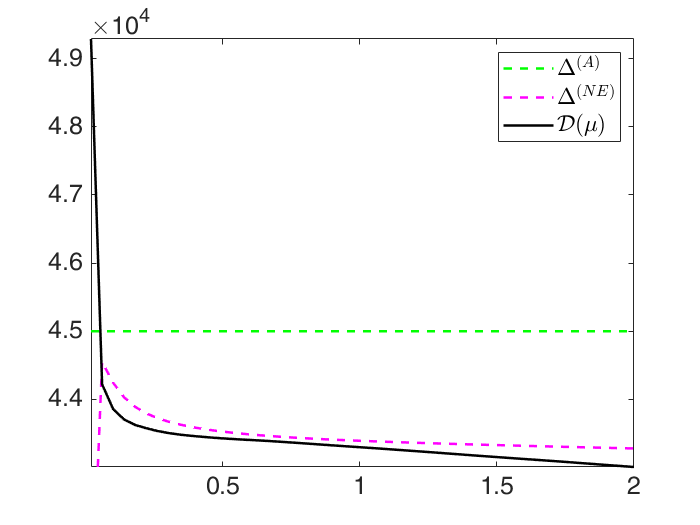}   &\includegraphics[height=3.9cm]{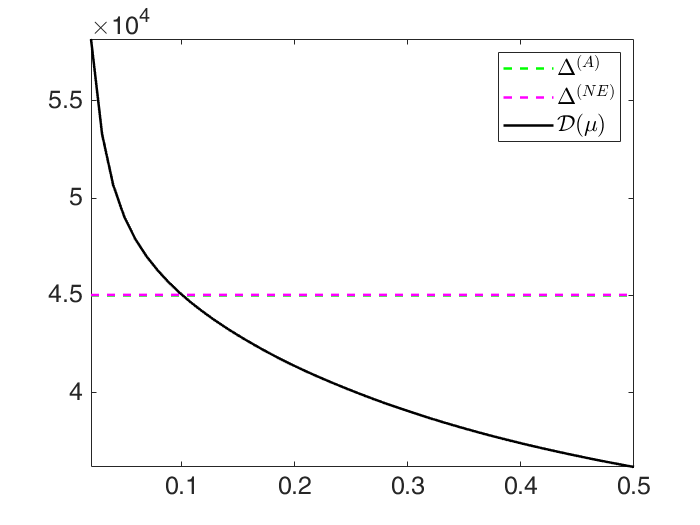} \\
		(a)&(b)\\
		\includegraphics[height=3.9cm]{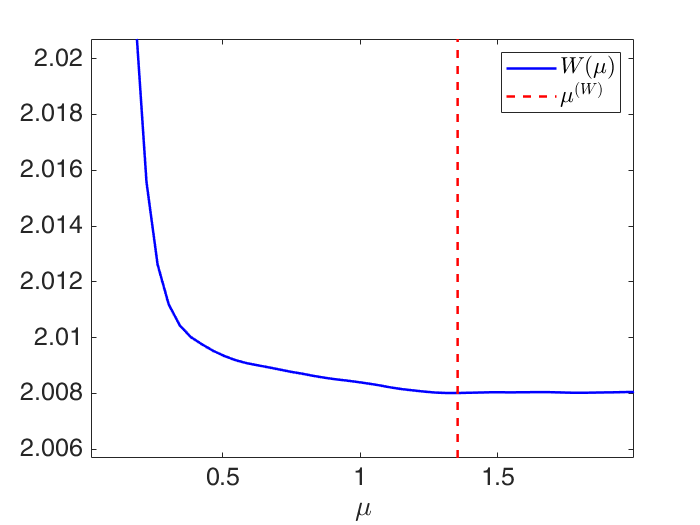}   &\includegraphics[height=3.9cm]{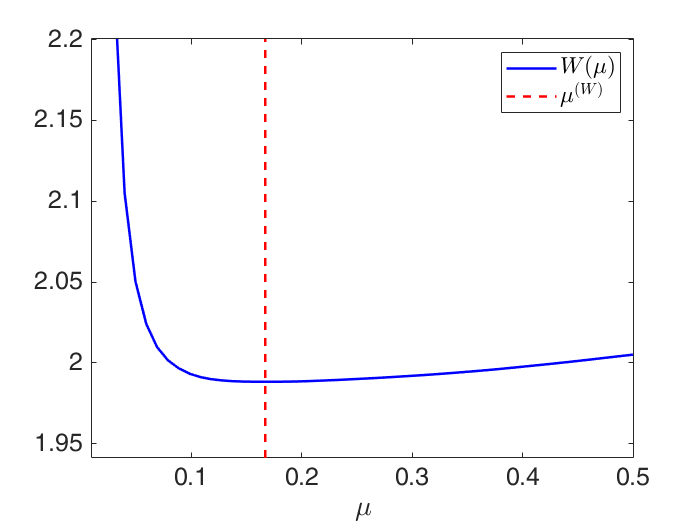} \\
		(c)&(d)\\
		\includegraphics[height=3.9cm]{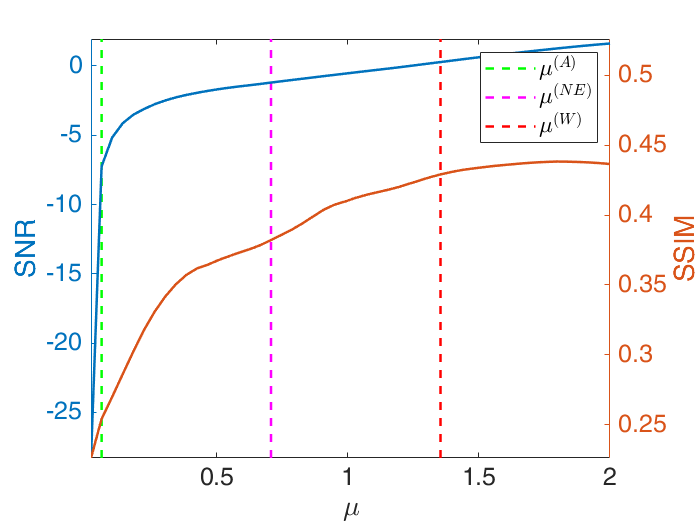}&
		\includegraphics[height=3.9cm]{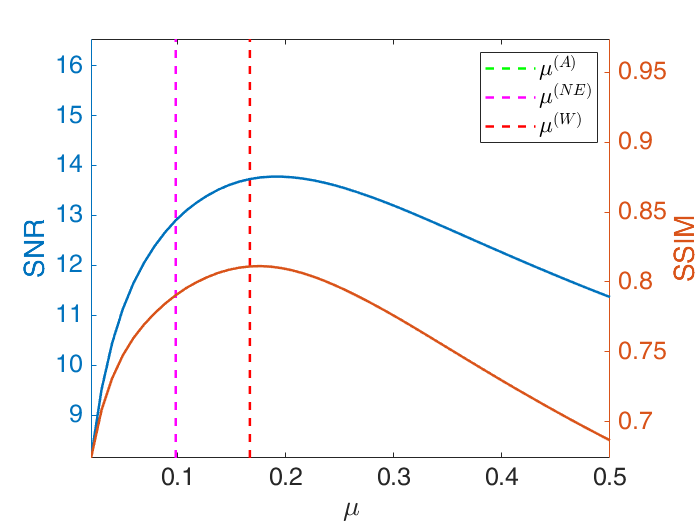} \\
		(e) & (f)
	\end{tabular}
	\setlength{\tabcolsep}{4pt}
	\begin{tabular}{r||rrr||rrr||rrr}
		&
		\multicolumn{3}{c}{ADP}&\multicolumn{3}{c}{NEDP}&\multicolumn{3}{c}{PWP}\\
		\hline\hline
		$I_0$& ${\mu}^{(A)}$&SNR&SSIM& ${\mu}^{(NE)}$&SNR&SSIM& ${\mu}^{(W)}$&SNR&SSIM\\
		\hline\hline
		1.5&0.060&-7.251&0.254&0.706&-1.200&0.382&1.853&\textbf{0.278}&\textbf{0.429}\\
		5&2.657&3.308&0.323&0.428&0.468&\textbf{0.432}&2.200&\textbf{3.866}&0.381\\
		10&2.428&3.432&0.311 &0.542&3.340&0.523&0.771&\textbf{4.472}&\textbf{0.534}\\
		20&1.302&\textbf{6.154}&0.466 &0.383&4.714&0.556&0.420&{5.040}&\textbf{0.562}\\
		50&0.516&8.162&\textbf{0.622} &0.320&7.636&0.612&0.589&\textbf{8.310}&0.614\\
		100&0.300 &\textbf{8.970} &\textbf{0.672} &0.257 &8.580 &0.666 &0.286 &8.854 &0.670 \\
		1000&0.098&12.901&0.790 &0.098&12.901&0.790&0.166&\textbf{13.728}& \textbf{0.811}\\
	\end{tabular}
	\caption{Test image \texttt{brain}. From top to bottom: discrepancy curves, whiteness curves and achieved SNR/SSIM for $I_{0}=1.5$ (left) and $I_{0}=1000$ (right). Output $\mu$- and SNR/SSIM values
		obtained by the ADP, the NEDP and the PWP for
		different $I_0$.
		%
	}
	\label{fig:mri_curve}
\end{figure}

\begin{figure}
	\centering
	\renewcommand{\arraystretch}{0.3}
	\setlength{\tabcolsep}{7pt}
	\begin{tabular}{cccc}
		$\bm{y}$&ADP& NEDP&PWP\\  
		\includegraphics[scale = 0.15]{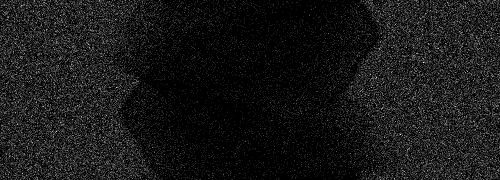} & \begin{overpic}[height=2.4cm]{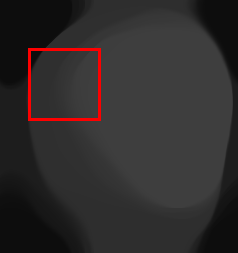}
			\put(60,1){\color{red}%
				\frame{\includegraphics[scale=0.13]{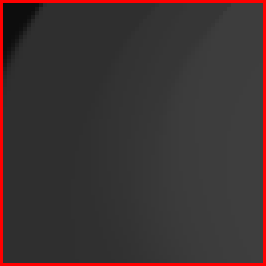}}}
		\end{overpic}   &  \begin{overpic}[height=2.4cm]{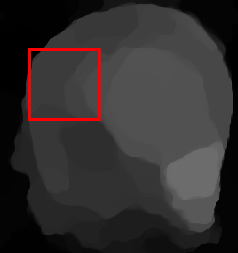}
			\put(60,1){\color{red}%
				\frame{\includegraphics[scale=0.13]{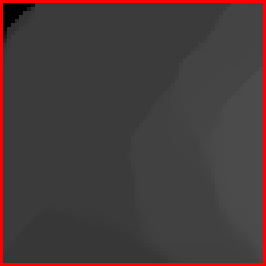}}}
		\end{overpic} &  \begin{overpic}[height=2.4cm]{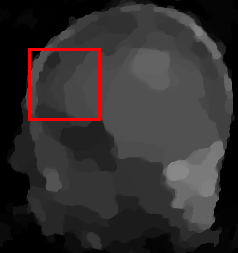}
			\put(60,1){\color{red}%
				\frame{\includegraphics[scale=0.13]{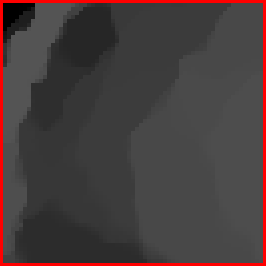}}}
		\end{overpic} \\
		\includegraphics[scale = 0.15]{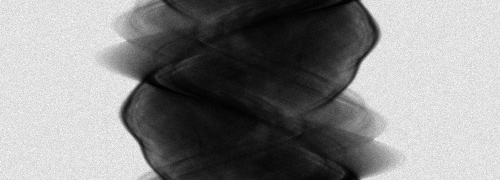} & \begin{overpic}[height=2.4cm]{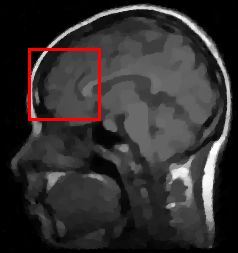}
			\put(60,1){\color{red}%
				\frame{\includegraphics[scale=0.13]{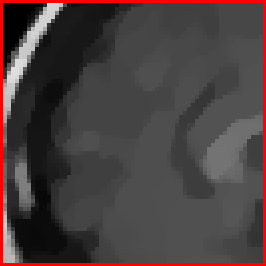}}}
		\end{overpic}   &  \begin{overpic}[height=2.4cm]{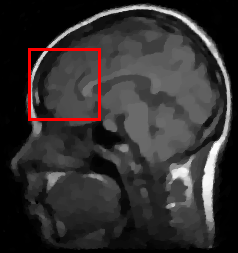}
			\put(60,1){\color{red}%
				\frame{\includegraphics[scale=0.13]{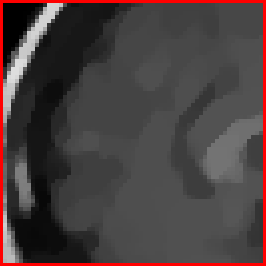}}}
		\end{overpic} &  \begin{overpic}[height=2.4cm]{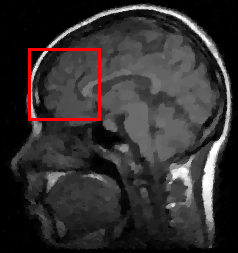}
			\put(60,1){\color{red}%
				\frame{\includegraphics[scale=0.13]{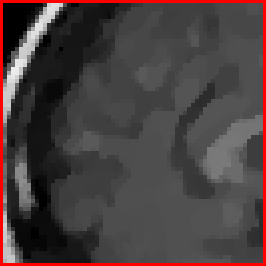}}}
		\end{overpic} 
	\end{tabular}
	\caption{Test image \texttt{brain}. From left to right: observed data $\bm{y}$, reconstruction using the ADP, NEDP and PWP for $I_{0}=1.5$ (top row) and $I_{0}=1000$ (bottom row).}
	\label{fig:mri}
\end{figure}

\section{Conclusion}
\label{sec:concl}
In this work, we have discussed the introduction of a novel parameter selection strategy in variational models under Poisson data corruption. Our proposal relies on the extension of the whiteness principle to a standardized version of the Poisson noise corrupted observations.
The derived Poisson Whiteness Principle has been tested on image restoration and CT reconstruction problems. In the latter case, we employed a linearized version of the ADMM which fasten the computations when the forward model operator does not present an advantegeous structure. The Poisson Whiteness Principle has been compared with the popular ADP and the NEDP, recently proposed by the same authors; the newly introduced approach has been shown to outperform the competitors especially in the lower-counting regimes.

\medskip

\noindent\textbf{Acknowledgements} All the authors are members of the ``National Group for
Scientific Computation (GNCS-INDAM)''. The research of FB, AL, FS has been funded by the ex60 project ``Funds for selected research topics'', while MP acknowledges the contribution of ``Young researchers funding'' awarded by GNCS-INDAM.

\bibliographystyle{plain}
\bibliography{mybibfile}

\end{document}